\documentclass[11pt]{article}
\usepackage{amsmath, amscd, amssymb, latexsym, epsfig, color, amsthm}
\usepackage[all]{xy}
\usepackage{verbatim}
\numberwithin{equation}{section}

\newtheorem{theorem}{Theorem}[section]
\newtheorem{proposition}[theorem]{Proposition}

\newtheorem{corollary}[theorem]{Corollary}
\newtheorem{conjecture}[theorem]{Conjecture}
\newtheorem{lemma}[theorem]{Lemma}

\theoremstyle{definition}
\newtheorem{definition}[theorem]{Definition}
\newtheorem{example}[theorem]{Example}

\newtheorem{remark}[theorem]{Remark}

\DeclareMathOperator{\skel}{Skel}
\DeclareMathOperator\lk{\mathrm{lk}}

\DeclareMathOperator\st{\mathrm{st}}

\DeclareMathOperator{\supp}{\mathrm{supp}}

\newcommand{\N}{{\mathbb N}}
\newcommand{\R}{{\mathbb R}}
\newcommand{\Q}{{\mathbb Q}}
\newcommand{\Z}{{\mathbb Z}}

\newcommand{\Soc}{\mathrm{soc}}
\newcommand{\Stress}{\mathcal{S}}
\newcommand{\F}{\mathbb{F}}

\title{Lower bounds on the $g$-numbers of spheres without large missing faces}
\author{
	Isabella Novik\thanks{Research of IN is partially\textsl{} supported by NSF grant  DMS-2246399.}\\
	\small Department of Mathematics\\[-0.8ex]
	\small University of Washington\\[-0.8ex]
	\small Seattle, WA 98195-4350, USA\\[-0.8ex]
	\small \texttt{novik@uw.edu}
	\and 
	Hailun Zheng\thanks{Research of HZ is partially\textsl{} supported by NSF grant DMS-2535689.} \\
	\small Department of Mathematics\\[-0.8ex]
	\small University of Hawai`i at M\={a}noa\\[-0.8ex]
	\small 2565 McCarthy Mall, Honolulu, HI 96822, USA \\[-0.8ex]
	\small \texttt{hailunz@hawaii.edu}
}
\begin{document}
\maketitle
\begin{abstract}
	We establish several new lower bounds on the $g$-numbers of simplicial spheres without large missing faces. For this class of spheres, we derive bounds on the $g$-numbers in terms of the independence numbers of their graphs, extending a result of Chudnovsky and Nevo. As a consequence, we show that flag $(d-1)$-spheres---and more generally, flag normal $(d-1)$-pseudomanifolds---satisfy $g_2\geq (1/2-\delta(d))f_0$, where $\delta(d)$ is a function of $d$ with $\delta(d)\to 0$ as $d\to \infty$.  We further prove that, for simplicial $(d-1)$-spheres without large missing faces, an initial segment of the $g$-vector forms a level sequence, yielding additional inequalities among the $g$-numbers. 
Finally, we show that simplicial $4$-spheres without missing faces of dimension greater than two satisfy $g_2\geq \frac{2}{5}f_0 - \frac{6}{5}$.
\end{abstract}

\section{Introduction}
In this paper, we establish new lower bounds on the face numbers of simplicial spheres without large missing faces. Throughout, by a sphere we mean a (simplicial) $\Z/2\Z$-homology sphere. We denote by $f_i(\Delta)$ the number of $i$-faces of a simplicial complex $\Delta$. The celebrated Lower Bound Theorem asserts that if $\Delta$ is a normal pseudomanifold of dimension $d-1\geq 3$, then $f_1(\Delta)-df_0(\Delta)+\binom{d+1}{2}\geq 0$, and moreover, equality holds if and only if $\Delta$ is a stacked sphere \cite{Barnette-LBT-pseudomanifolds, Kalai87, Fogelsanger88, Tay}. This theorem is usually phrased in terms of the $g$-numbers---certain alternating sums of the face numbers. In this language, the theorem states that for any normal pseudomanifold $\Delta$ of dimension at least three, $g_2(\Delta)\geq 0$, with equality if and only if $\Delta$ is a stacked sphere.

In the same spirit, the Generalized Lower Bound Theorem (namely, the linear inequalities of the $g$-theorem) asserts that if $\Delta$ is a $(d-1)$-sphere, then for all $1\leq i\leq d/2$, we have $g_i(\Delta)\geq 0$ \cite{Stanley80, Adiprasito-g-conjecture, AdiprasitoPapadakisPetrotou, PapadakisPetrotou, KaruXiao}. Furthermore, equality $g_i(\Delta)=0$ for some $1\leq i\leq d/2$ holds if and only if $\Delta$ is $(i-1)$-stacked; see \cite{McMullenWalkup71, MuraiNevo2013}. The notion of $(i-1)$-stackedness generalizes that of stackedness; instead of giving precise definitions, we merely note that a stacked sphere always has missing faces of dimension $\geq d-1$, while an $(i-1)$-stacked sphere has missing faces of dimension $\geq d-i+1$.

This makes it plausible to conjecture that spheres without large missing faces should satisfy significantly stronger lower-bound--type results. Of particular interest are flag spheres---these are spheres without missing faces of dimension larger than one. A motivation for their study comes from their surprising metric properties first noticed by Gromov; see \cite{Gromov87}. Building on Gromov's results, and motivated by the Hopf conjecture in Riemannian geometry, Charney and Davis \cite{CharneyDavis95} posited the following purely combinatorial conjecture:

\begin{conjecture}[Charney--Davis Conjecture]
	Let $k\geq 2$. Then all flag $(2k-1)$-spheres satisfy $g_k-g_{k-1}+\dots +(-1)^kg_0\geq 0$.
\end{conjecture}

At present, this conjecture is known to hold only in the case $k=2$ \cite{DavisOkun}, and it remains wide open in all higher-dimensional cases. 
Inspired by the Charney--Davis Conjecture and the Generalized Lower Bound Conjecture, Gal \cite{Gal05} defined the $\gamma$-numbers $\gamma_0, \gamma_1, \dots, \gamma_{\lfloor d/2\rfloor}$ of a $(d-1)$-sphere as certain weighted alternating sums of the $h$-numbers. Consequently, the $\gamma$-numbers can be expressed as linear combinations of the $g$-numbers. For instance, when $d=2k$, one has $\gamma_k=g_k-g_{k-1}+\dots +(-1)^kg_0$.  Gal conjectured that the $\gamma$-numbers of {\em flag} spheres are all nonnegative. Very little progress has been made on Gal's conjecture. In a very recent paper  of Chudnovsky and Nevo \cite{Chudnovsky-Nevo}, it was shown that for flag $(d-1)$-spheres, $g_2\geq \frac{1-\delta(d)}{2d+1}f_0$, where $\delta(d)\to 0$ as $d\to \infty$.

We denote by $S(j,d-1)$ the collection of ($\mathbb{Z}/2\mathbb{Z}$-homology) $(d-1)$-spheres without missing faces of dimension larger than $j$. 
In particular, $S(d,d-1)$ is the class of all $(d-1)$-spheres, while $S(1,d-1)$ is the class of all flag $(d-1)$-spheres.
For the more general cases of spheres in $S(j, d-1)$ with $1\leq j\leq d$, Nevo proposed a conjecture that interpolates between Gal's Conjecture for $S(1, d-1)$ and the generalized Lower Bound Theorem for $S(d, d-1)$; see \cite[Conjecture 1.5]{Nevo2009}. For instance, for spheres in $S(2,4)$, his conjecture states that $g_2\geq g_1=f_0-6$.

While we are still very far from proving Gal’s or Nevo’s conjectures, in this paper we use the affine stress spaces and the Stanley--Reisner rings to establish several new lower bounds on the $g$-numbers of spheres in $S(j, d-1)$. Our results can be summarized as follows:
\begin{enumerate}
	\item We prove that any sphere in $S(2, 4)$ satisfies $g_2\geq \frac{2}{5}f_0 - \frac{6}{5}$; see Theorem \ref{thm: S(2,4)}. This improves the bound given in \cite[Lemma 4.2]{Nevo2009}.
	\item We establish a lower bound on the $g_2$-numbers of normal pseudomanifolds without large missing faces, and, more generally, lower bounds on the $g$-numbers of spheres without large missing faces in terms of the independence numbers of their graphs; see Theorems \ref{thm: flag lower bound II}, \ref{thm: lower bound on g II}, and \ref{thm: g_{k+1} in terms of f_{k-1}}. In particular, this implies that flag normal $(d-1)$-pseudomanifolds satisfy $g_2\geq (1/2-\delta(d))f_0$, where $\delta(d)$ is a function of $d$ with $\delta(d)\to 0$ as $d\to \infty$. 
	\item Using results on level rings, we obtain several additional inequalities on the $g$-numbers; see Corollary \ref{thm: level ring}. As a consequence, we produce counterexamples to the conjectures on affine stress spaces raised in \cite{NZ-Aff-Reconstr}; see Corollaries \ref{counterexample1} and \ref{counterexample2}.
\end{enumerate}

The structure of the paper is as follows. Section 2 reviews background on face enumeration for spheres and the relevant Stanley--Reisner ring theory. Section 3 discusses affine stress theory together with the cone lemma and some of its applications. Section 4 focuses on spheres in $S(2,4)$, where we establish the lower bound in Theorem~\ref{thm: S(2,4)}. In Section 5, we investigate the relationship between the $g$-numbers of spheres and the independence numbers of the associated graphs. Finally, Section 6 derives additional relations among the $g$-numbers using algebraic tools and presents counterexamples to two conjectures.

	\section{Preliminaries}
	\subsection{Simplicial complexes and spheres}
	An (abstract) {\em simplicial complex} $\Delta$ with vertex set $V=V(\Delta)$ is a nonempty collection of subsets of $V$ that is closed under inclusion and contains all singletons; that is, $\{v\}\in\Delta$ for all $v\in V$. Given a $(d+1)$-set $V$, the collection of all subsets of $V$ is a {\em $d$-simplex}, which we usually denote by $\overline{V}$, or $\sigma^d$ when we do not need to emphasize the vertex set. Similarly, the collection of all subsets of $V$ except $V$ itself is the {\em boundary complex} of a $d$-simplex, denoted by $\partial\overline{V}$ or $\partial \sigma^d$.
	
	The elements of a simplicial complex $\Delta$ are called {\em faces} of $\Delta$. A face $\tau$ of $\Delta$ has {\em dimension} $i$ if $|\tau|=i+1$; in this case we say that $\tau$ is an {\em $i$-face}. We usually refer to $0$-faces as {\em vertices}, $1$-faces as {\em edges}, and the maximal under inclusion faces as {\em facets}. For brevity, we sometimes denote an $i$-face $\{a_1, a_2, \dots, a_{i+1}\}$ by $a_1a_2\dots a_{i+1}$. The {\em dimension of $\Delta$} is $\max\{\dim \tau: \tau\in \Delta\}$. A set $\tau\subseteq V$ is a {\em missing face} of $\Delta$ if $\tau$ is not a face of $\Delta$, but every proper subset of $\tau$ is a face of $\Delta$; it is a {\em missing $i$-face} if $|\tau|=i+1$.
	
	Let $\tau$ be a face of $\Delta$. The {\em star} and {\em link} of $\tau$ are defined as
	$$\st(\tau)=\st(\tau,\Delta)=\{\sigma \in \Delta \ : \  \sigma\cup \tau\in\Delta\} \;\;\text{ and }\;\;\lk(\tau)=\lk(\tau,\Delta)= \{\sigma\in \st(\tau) \ : \ \sigma\cap \tau=\emptyset\}.$$ When $\tau=v$ is a vertex, we also define $\Delta\backslash v=\{\sigma\in\Delta : v\notin\sigma\}$; this subcomplex of $\Delta$ is called the {\em antistar} of $v$.
	
	A subcomplex of $\Delta$ is called {\em induced} if it is of the form $\Delta[W]=\{\tau\in \Delta: \tau\subseteq W\}$ for some $W\subseteq V(\Delta)$. The subcomplex of $\Delta$ consisting of all faces of $\Delta$ of dimension $\leq k$ is called the {\em $k$-skeleton} of $\Delta$ and is denoted $\skel_k(\Delta)$; the $1$-skeleton of $\Delta$ is also known as the {\em graph} of $\Delta$. Finally, if $\Delta$ and $\Gamma$ are two simplicial complexes on disjoint vertex sets, then their {\em join} is 
	$$\Delta*\Gamma=\{\sigma\cup \tau: \sigma\in \Delta, \tau \in \Gamma\}.$$ If $\Gamma$ is a $0$-simplex, that is, $\Gamma=\{v,\emptyset\}$, we write $\Delta*\Gamma=\Delta*v$  and call this complex the {\em cone} over $\Delta$ with apex $v$. If $\Gamma$ is $\partial \sigma^1$ (that is, the complex consisting of two disjoint vertices), we write $\Delta*\Gamma=\Sigma \Delta$ and call $\Sigma\Delta$ the {\em suspension} of $\Delta$. 
	
	Let $\F$ be a field. A pure $(d-1)$-dimensional simplicial complex is an {\em $\F$-homology sphere} if, for every face $\sigma$ (including the empty face), the link $\lk(\sigma)$ has the homology of a $(d-1-|\sigma|)$-dimensional sphere (over $\F$). 
	Denote by $\|\Delta\|$ the geometric realization of $\Delta$. We say that $\Delta$ is a {\em simplicial $(d-1)$-sphere} if $\|\Delta\|$ is
	homeomorphic to a $(d-1)$-dimensional sphere. It is known that a $\Z/2\Z$-homology sphere is always an $\R$-homology sphere; see \cite[Lemma 2.1]{KaruXiao}. Moreover, a simplicial sphere is an $\F$-homology sphere for any field $\F$. The boundary complex of a simplicial polytope $P$, $\partial P$, is a simplicial $(d-1)$-sphere; such spheres are called {\em polytopal}. 
	
	In this paper, we work with the class of $\Z/2\Z$-homology spheres, which we refer to simply as spheres. We define $S(i, d-1)$ to be the set of $\Z/2\Z$-homology $(d-1)$-spheres whose missing faces all have dimension at most $i$. For example, if we set $K(i, d-1):=\partial \sigma^i* \partial \sigma^i *\partial \sigma^{d-2i}$ for $d/3\leq i\leq d/2$, then $K(i, d-1)\in S(i, d-1)$. Moreover, the class of flag $(d-1)$-spheres is precisely $S(1, d-1)$. We also note that the link of a $(j-1)$-face of a sphere in $S(i, d-1)$ is an element of $S(i, d-j-1)$.

	Let $\Delta$ be a $(d-1)$-sphere. Denote by $f_i=f_i(\Delta)$ the number of $i$-faces of $\Delta$. The {\em $f$-vector} of $\Delta$ is $(f_{-1},f_0, \dots, f_{d-1})$, and the {\em $h$-vector} $h(\Delta)=(h_0, h_1,\ldots,h_d)$ is defined by the identity
	$$\sum_{i=0}^df_{i-1}(t-1)^{d-i}=\sum_{j=0}^d h_it^{d-i}.$$ 
	The {\em $h$-polynomial }of $\Delta$ is $h(\Delta, t)=\sum_{i=0}^d h_it^i$. For example, the $h$-vector of the boundary complex of a $d$-simplex is $(1,1, \dots, 1)$, and consequently, $h(\partial \sigma^d, t)=1+t+\dots +t^d$.		
	The {\em $g$-vector} of $\Delta$, $g(\Delta)=(g_0,g_1,\ldots, g_{\lceil d/2\rceil})$, is defined by letting $g_0=1$ and $g_j=h_j-h_{j-1}$ for $1\leq j\leq \lceil d/2\rceil$. The Dehn--Sommerville relations \cite{Klee64} assert that the $h$-vector of a $(d-1)$-sphere is symmetric: $h_i=h_{d-i}$ for all $0\leq i\leq d$. In particular, it follows that if $d$ is odd, then $g_{\lceil d/2\rceil}=0$.
	
	We will also need the notion of a normal pseudomanifold, a class of simplicial complexes that significantly generalizes spheres. A simplicial complex $\Delta$ is a {\em normal $(d-1)$-pseudomanifold} if all facets have dimension $d-1$, every $(d-2)$-face is contained in exactly two facets, and the link of each face of dimension at most $d-3$ is connected. The $f$-numbers and $h$-numbers of a normal $(d-1)$-pseudomanifold are defined in the same way as for spheres; in particular, we will also be interested in $g_2=h_2-h_1=f_1-df_0+\binom{d+1}{2}$.
	
	\subsection{Graphs and the independence number}
	Let $G=(V, E)$ be a graph. If the vertex set is $V=\{v_1, v_2, \dots, v_\ell\}$ and the edge set is $E=\{v_iv_{i+1}: 1\leq i\leq \ell-1\}$, then the graph is a {\em path}, which we usually denote by $(v_1, v_2,\dots, v_\ell)$. If $v_\ell=v_1$ (and all other vertices are distinct), we say that the graph is a {\em cycle}.
	
	The {\em independence number} $\alpha(G)$ of $G$ is the maximum cardinality of an independent set in $G$. The independence number has been extensively studied in graph theory. One of the most important classical results in this area is the following; see, for instance, \cite[Chapter 41]{EignerZiegler}.
	\begin{theorem}[Tur\'an's theorem]
		Let $G$ be a graph with $f_0$ vertices and $f_1$ edges. Then $$\alpha(G)\geq \frac{f_0}{2f_1/f_0+1}.$$
	\end{theorem}
	When $\Delta$ is a simplicial complex, we write $\alpha(\Delta)$ for the independence number of its graph.
	
	\subsection{The Stanley--Reisner ring theory}
	
	For a $(d-1)$-dimensional simplicial complex $\Delta$ with vertex set $V=V(\Delta)$, let $\R[X]=\R[x_v : v\in V]$ be the polynomial ring over $\R$ with one variable for each vertex of $\Delta$.  The {\em Stanley--Reisner ideal} of $\Delta$ is the ideal of $\R[X]$ generated by the monomials corresponding to the missing faces of $\Delta$:
	\[I_\Delta=(x_{j_1}x_{j_2}\cdots x_{j_k} : \{j_1,\ldots,j_k\} \mbox{ is a missing face of }\Delta).\] The {\em Stanley--Reisner ring} (or {\em face ring}) of $\Delta$ is the quotient $\R[\Delta]:=\R[X]/I_\Delta$. This is a graded ring, and we write  $\R[\Delta]_i$ for its $i$th graded component. (We will occasionally consider a subfield $\F$ of $\R$ and write $\F[\Delta]:=\F[X]/I_\Delta$ for the Stanley--Reisner ring of $\Delta$ over $\F$. The discussion below applies in this setting as well.)
	
	Assume that $\Delta$ is a $\Z/2\Z$-homology $(d-1)$-sphere. An embedding of $\Delta$ in $\R^d$ is a function $p: V(\Delta)\to\R^d$. In this paper we usually work with a {\em generic embedding}. That is, choose $\{a_{v,j}: v\in V, j\in[d]\}\subset\R$, where $[d]: = \{1,2, \dots, d\}$, to be a set that is {\em algebraically independent} over $\Q$ and define $p(v)=(a_{v,1}, \dots, a_{v, d})$ for each vertex $v\in V(\Delta)$. Occasionally, when $\Delta$ is realizable as the boundary complex of a simplicial $d$-polytope $P\subset \R^d$, we also consider a {\em natural embedding} of $\Delta$. In this case, we take  $p$ to be the map assigning to each vertex $v\in \Delta$ the coordinates of the corresponding vertex of $P$. 
	
	Given an embedding $p$ of $\Delta$, define $\theta_j:=\sum_{v\in V} p(v)_j x_v$ for $1\leq j\leq d$. When $p$ is a generic or natural embedding of $\Delta$, the quotient ring $\R[\Delta]/(\theta_1, \dots, \theta_d)$ is a finite-dimensional $\R$-vector space. In this case, the sequence $\Theta=\Theta(p)=(\theta_1, \dots, \theta_d)$ is called a {\em linear system of parameters} (l.s.o.p.) and $\R[\Delta]/(\Theta)$ is called an {\em Artinian reduction} of $\R[\Delta]$. 

	Since $\Delta$ is a $(d-1)$-sphere, the ring $\R[\Delta]$ is Gorenstein. Consequently, any Artinian reduction of $\R[\Delta]$ is a Poincar\'e duality algebra, with $\dim_\R \big(\R[\Delta]/(\Theta))_i=h_i(\Delta)$ for all $0\leq i\leq d$. Furthermore, letting $c=\sum_{v\in V} x_v$, the Hard Lefschetz theorem for polytopes and spheres \cite{Stanley80,McMullen96,Adiprasito-g-conjecture,PapadakisPetrotou, AdiprasitoPapadakisPetrotou,KaruXiao} implies that $\R[\Delta]/(\Theta, c)$ is also a finitely generated standard graded algebra, with $\dim_\R (\R[\Delta]/(\Theta,c))_i=g_i(\Delta)$ for all $0\leq i \leq \lceil d/2\rceil$ and all other graded components of $\R[\Delta]/(\Theta,c)$ vanish.
	
	Macaulay's theorem (see \cite[Theorem II.2.2]{Stanley96}) asserts that Hilbert functions of standard graded algebras are  $M$-sequences. To define an $M$-sequence, note that for all positive integers $n$ and $i$, there is a unique expression $$n=\binom{n_i}{i}+\binom{n_{i-1}}{i-1}+\dots+\binom{n_j}{j}, \quad \mbox{where } n_i>n_{i-1}>\dots>n_j\geq j\geq 1.$$ Define $n^{\langle i\rangle}:=\binom{n_i+1}{i+1}+\binom{n_{i-1}+1}{i}+\dots+\binom{n_j+1}{j+1}$ for $n>0$, and set $0^{\langle 0 \rangle}:=0$. A sequence of nonnegative integers $(a_0,a_1,\dots, a_m)$ is called an {\em $M$-sequence} if $a_0=1$ and $0\leq a_{i+1}\leq a_i^{\langle i \rangle}$ for all $1\leq i\leq m-1$. A weaker, but often more convenient, property of an $M$-sequence $(a_0,a_1,\dots, a_m)$ is the following: writing $a_i=\binom{x_i}{i}$ for a real number $x_i$, one has $0\leq a_{i+1}\leq \binom{x_i+1}{i+1}$ for all $1\leq i\leq m-1$.

	An immediate consequence of the Hard Lefschetz theorem and Macaulay's theorem is:
	
\begin{theorem} {\rm ($g$-theorem)}
Let $\Delta$ be a $(d-1)$-sphere. Then $(g_0(\Delta),g_1(\Delta),\dots,g_{\lfloor d/2 \rfloor}(\Delta))$ is an $M$-sequence.
\end{theorem}

The part of the $g$-theorem asserting that the $g$-numbers are nonnegative, together with the characterization of spheres satisfying $g_i=0$ for a given $i\leq \lfloor d/2\rfloor$ (see \cite{McMullenWalkup71,MuraiNevo2013}), is known as the Generalized Lower Bound Theorem. Recall that a $(d-1)$-sphere $\Delta$ is called {\em $(i-1)$-stacked }if there exists an $\R$-homology $d$-ball $B$ such that $\partial B=\Delta$ and $\skel_{d-i}(B)=\skel_{d-i}(\Delta)$. (In other words, all interior faces of $B$ have dimension at least $d-i+1$.)

\begin{theorem}\label{GLBT}
		Let $\Delta$ be a $(d-1)$-sphere. Then $g_i(\Delta)\geq 0$ for all $1\leq i\leq \lfloor \frac{d}{2}\rfloor$. Furthermore, $g_i(\Delta)=0$ for some $1\leq i\leq \lfloor \frac{d}{2}\rfloor$ if and only if $\Delta$ is $(i-1)$-stacked. 
\end{theorem}
	
	\subsection{Corollaries on the $g$-numbers and $\gamma$-numbers}

	Theorem \ref{GLBT} shows that a $(d-1)$-sphere with $g_i=0$ for some $i\leq \lfloor d/2\rfloor$ does not belong to $S(d-i,d-1)$. This immediately implies the following.
	\begin{corollary}\label{S(d-i, d-1) implies g_i>0}
		If $i\leq \lfloor d/2\rfloor$ and $\Delta\in S(d-i,d-1)$, then $g_i(\Delta)\geq 1$.
	\end{corollary}
	
	In the case $i=2$, Nevo and Novinsky provided the following characterization of spheres in $S(d-2, d-1)$ with $g_2=1$; see \cite{NevoNovinsky}.
		\begin{theorem} \label{thm:Nevo-Novinsky}
			Let $d\geq 4$. Suppose $\Delta\in S(d-2, d-1)$ satisfies $g_2(\Delta) = 1$. Then either $\Delta=\partial \sigma^i *\partial \sigma^{d-i}$  for some $2\leq i\leq d-2$, or $\Delta= C* \partial \sigma^{d-2}$, where $C$ is a cycle.
	\end{theorem}
	
	McMullen's integral formula provides additional relations between the $g$-numbers of a pure simplicial complex  and the $g$-numbers of its vertex links; see \cite[p.~183]{McMullen70} and \cite[Proposition 2.3]{Swartz05}. Recall that when $d$ is odd, every $(d-1)$-sphere satisfies $g_{\lceil d/2\rceil}=0$. 
	\begin{lemma}\label{McMullen's formula}
	Let $\Delta$ be a pure simplicial complex of dimension $d-1$. Then for every $0\leq k\leq \lfloor \frac{d-1}{2}\rfloor$,
\[\sum_{v\in V(\Delta)} g_k(\lk(v))=(k+1)g_{k+1}(\Delta)+(d+1-k)g_k(\Delta).\]
	\end{lemma}
	
\begin{corollary}\label{cor: S(k, 2k)}
		Let $k\geq 1$ and let $\Delta\in S(k, 2k)$. Then $g_k(\Delta)\geq \frac{f_0(\Delta)}{k+2}$. In particular, if $\Delta\in S(2,4)$, then $g_2(\Delta)\geq \frac{1}{4}f_0(\Delta)$.
	\end{corollary}
	\begin{proof}
		The vertex links of $\Delta$ lie in $S(k, 2k-1)$. Hence, by Corollary \ref{S(d-i, d-1) implies g_i>0}, we have $g_k(\lk(v))\geq 1$ for all $v\in \Delta$. Applying McMullen's formula and using the fact that $g_{k+1}(\Delta)=0$, we obtain $$g_k(\Delta)=\frac{1}{k+2}\sum_{v\in V(\Delta)} g_k(\lk(v))\geq \frac{f_0(\Delta)}{k+2}.$$
	\end{proof}
	
	Nevo \cite{Nevo2009} conjectured a significantly stronger lower bound on the $g_2$-numbers of spheres in $S(2, 4)$.
	\begin{conjecture}\label{conj: S(2,4)}
		Let $\Delta\in S(2,4)$. Then $g_2(\Delta)\geq g_1(\Delta)$.
	\end{conjecture}
	
	We conclude this section with a discussion of the face numbers of flag spheres. In this setting, the $\gamma$-numbers may be viewed as analogs of the $g$-numbers for general spheres. In particular, the following analog of McMullen's formula holds for the $\gamma$-numbers. Although the proof follows standard techniques, we were unable to find this result explicitly stated in the literature, and so we include it here for completeness.

Let $\Delta$ be a $(d-1)$-sphere. The $\gamma$-numbers of $\Delta$ are defined by the relation $$h(\Delta,t)=\sum_{i=0}^d h_i(\Delta)t^i=\sum_{k=0}^{\lfloor d/2\rfloor} \gamma_k(\Delta) t^k(1+t)^{d-2k}.$$
For example, one easily checks that $\gamma_0=1$, $\gamma_1=f_0-2d$, and $\gamma_2=f_1-(2d-3)f_0+2d(d-2)$.
Recall that, under the standard $\N$-grading (where $\N$ denotes the set of all nonnegative integers), the Hilbert series\footnote{For background on Hilbert series of Stanley--Reisner rings, see Sections II.1 and II.2 of Stanley's book \cite{Stanley96}.} $F(\Delta,t)$ of the Stanley--Reisner ring $\R[\Delta]$ is given by $h(\Delta, t)/(1-t)^d$.

	\begin{lemma} \label{lem:gamma-analog-of-McMullen}
		Let $\Delta\in S(d, d-1)$. Then for every $0\leq i\leq \lfloor (d-1)/2 \rfloor$, $$\sum_{v\in V(\Delta)} \gamma_i(\lk(v))=(i+1)\gamma_{i+1}(\Delta)+(2d-4i)\gamma_i(\Delta).$$
	\end{lemma}
	
\begin{proof}
		Assume that $V(\Delta)=[n]$.  Consider first the $\N^n$-graded Hilbert series of $\R[\Delta]$. For  $a=(a_1, \dots, a_n)\in \N^n$,  write $t_1^{a_1}\dots t_n^{a_n}=\mathbf{t}^a$ and $\deg \mathbf{t}^a=a$. Then
$$F(\Delta, t_1, \dots, t_n)=\sum_{a\in \N^n} \dim \R[\Delta]_{a} \mathbf{t}^a=\sum_{\sigma\in \Delta}\prod_{v\in \sigma} \frac{t_v}{1-t_v}.$$
Applying the differential operator $\partial_c=\sum_{v=1}^n \partial t_v$, we obtain:
$$\partial_c F(\Delta, t_1, \dots, t_n)=
	\sum_{v=1}^n \frac{1}{(1-t_v)^2}\sum_{\tau\in \lk(v)} \prod_{j\in \tau}\frac{t_j}{1-t_j}.$$
Setting $t_1=\dots =t_n=t$, we infer  $$\frac{\mathrm{d}}{\mathrm{d}t}F(\Delta, t)=\frac{1}{(1-t)^2}\sum_{v=1}^n \frac{h(\lk(v), t)}{(1-t)^{d-1}}=\sum_{v=1}^n \frac{\sum_{k=0}^{\lfloor (d-1)/2\rfloor}\gamma_k(\lk(v))t^k(1+t)^{d-1-2k}}{(1-t)^{d+1}}.$$
On the other hand, differentiating $F(\Delta, t)=\frac{h(\Delta, t)}{(1-t)^d}=\frac{\sum_{i=0}^{\lfloor d/2\rfloor} \gamma_i(\Delta)t^i(1+t)^{d-2i}}{(1-t)^d}$ with respect to $t$ yields
		$$\frac{\mathrm{d}}{\mathrm{d}t}F(\Delta, t)=\frac{\sum_{i=0}^{\lfloor d/2\rfloor} \gamma_i(\Delta)t^{i-1}(1+t)^{d-2i-1}(i+(2d-2i)t+it^2)}{(1-t)^{d+1}}.$$
Note that $i+(2d-2i)t+it^2=(2d-4i)t+i(1+t)^2$. Comparing coefficients of $t^i(1+t)^{d-1-2i}$ in the two expressions for $\frac{\mathrm{d}}{\mathrm{d}t}F(\Delta, t)$, we conclude that  $\sum_{v=1}^n \gamma_i(\lk(v))=(i+1)\gamma_{i+1}(\Delta)+(2d-4i)\gamma_i(\Delta)$ for all $1\leq i\leq \lfloor (d-1)/2\rfloor$.
	\end{proof}
	
	By definition, a $2k$-sphere satisfies $\gamma_{k+1}=0$. Together with Lemma \ref{lem:gamma-analog-of-McMullen}, this yields the following result; see also \cite[Corollary 2.2.2]{Gal05}.
	\begin{corollary}
		If $\gamma_k\geq 0$ holds for all flag $(2k-1)$-spheres, then $\gamma_k\geq 0$ holds for all flag $2k$-spheres. In particular, the Davis--Okun theorem \cite{DavisOkun} implies that $\gamma_2\geq 0$ holds for all flag $4$-spheres. Furthermore,  all flag $5$-spheres satisfy $3\gamma_3+4\gamma_2\geq 0.$
	\end{corollary}
	
	\section{The stress theory}
	\subsection{The affine stress spaces}	
	Continuing with the notation of Section 2, let $\Delta$ be a simplicial complex of dimension $d-1$ and let $p:V=V(\Delta)\to\R^d$ be an embedding. Let $X =\{x_v : v\in V\}$, let $\F\subseteq \R$ be a field that contains all coordinates $\{p(v)_j : v\in V, j\in [d]\}$, and let $\F[X]$ be the corresponding polynomial ring.

For each $v\in V$, consider the differential operator  $\partial_{x_v}:=\frac{\partial}{\partial{x_v}}$ acting on $\F[X]$. More generally, if $\mu=x_{i_1}\cdots x_{i_s}\in \F[X]$ is a monomial, define $\partial_\mu : \F[X] \to \F[X]$ by $\rho \mapsto \partial_{x_{i_1}}\cdots\partial_{x_{i_s}}\rho$. If $\ell =\sum_{v\in V} \ell_v x_v$ is a linear form in $\F[X]$, define
	$$\partial_{\ell} : \F[X]\to\F[X] \quad \mbox{by} \quad \rho \mapsto \sum_{v\in V}\ell_v\cdot\partial_{x_v}\rho=\sum_{v\in V}\ell_v\frac{\partial \rho}{\partial{x_v}}.$$
	For a monomial $\mu\in\F[X]$, we define its {\em support} by $\supp(\mu)=:\{v\in V : x_v|\mu\}$.

	\begin{definition}
		A homogeneous polynomial $\lambda=\lambda(X)=\sum_\mu \lambda_\mu \mu\in\F[X]$ of degree $k$ is called an {\em affine $k$-stress} on $(\Delta, p)$ if it satisfies the following conditions:
		\begin{itemize}
			\item Every nonzero term $\lambda_\mu \mu$ of $\lambda$ is supported on a face of $\Delta$; that is, $\supp(\mu)\in\Delta$.
			\item $\partial_{\theta_i}\lambda=0$ for all $i=1,\ldots, d$.
			\item $\partial_{c}\lambda=\sum_{v\in V}\partial_{x_v}\lambda=0$.
		\end{itemize}
	We say that a face $\tau$ {\em participates} in a stress $\omega$, if there is a nonzero monomial $\mu$ of $\omega$ with $\tau\subseteq \supp(\mu)$. The {\em support} of $\omega$ is the subcomplex of $\Delta$ consisting of all faces that participate in $\omega$.
	\end{definition}
	
	The set of affine $k$-stresses on $\Delta$ forms a vector space over $\F$, which we denote by $\Stress^a_k(\Delta, p;\F)$. When the context is clear, we sometimes omit $p$ and $\F$ from the notation. The stress space $\Stress^a_i(\Delta,p;\F)$ coincides with a certain graded component of the Macaulay inverse system of $I_\Delta+(\Theta(p),c)\subseteq \F[X]$; see \cite{MNZ} for more details. In particular, the Hard Lefshetz theorem for polytopes and spheres admits the following (weaker) restatement in the language of stress spaces. 

	\begin{theorem}\label{thm: Lefschetz}
		Let $(\Delta, p)$ be a $(d-1)$-sphere with a generic or natural embedding $p$. Then for every $1\leq k\leq \lfloor d/2\rfloor$, we have $\dim_\F \Stress_k^a(\Delta,p;\F)=g_k(\Delta).$ Here $\F\subseteq \R$ is any field containing $\Q(p(v)_j : v\in V(\Delta), j\in [d])$.
	\end{theorem}	
	
	If $\Lambda$ is a full-dimensional subcomplex of $\Delta$, then $\Stress^a_k(\Lambda,p)\subseteq \Stress^a_k(\Delta, p)$. In other words, any affine stress $\omega\in \Stress^a_k(\Lambda,p)$ is also an affine stress on $(\Delta,p)$ whose support is contained in $\Lambda$. Throughout the paper, we frequently apply this observation tacitly in the case $\Lambda=\st(\tau)$ for $\tau\in \Delta$.
	
	\subsection{The cone lemma and its applications}
	
Many proofs in this paper rely on the following variant of the cone lemma from \cite[Lemmas 3.1 and 3.2]{NZ-g_i}. The proof is identical to that in \cite{NZ-g_i}, and so we omit it.
	\begin{lemma}\label{cone lemma}
		Let $\Delta$ be a $(d-2)$-dimensional simplicial complex with $V(\Delta)=[n]$, and let $\Gamma=0*\Delta$. Let $\{a_{u, j}, b_u: u\in[n], j\in [d-1]\}\subset \R$ be a set of numbers that are algebraically independent over $\Q$. Define $\F'=\Q\big(\frac{a_{u,j}}{b_{u}} : u\in [n], j\in [d-1]\big)$ and $\F=\Q(a_{u,j}, b_u: u\in [n],j\in[d-1])=F'(b_1,\ldots, b_n)$. Embed $\Delta$ and $\Gamma$ in $(\F')^{d-1}$ and $\F^d$, respectively, using maps $p'$ and $p$ defined by $$p'(u)=\left(\frac{a_{u,1}}{b_{u}},\dots,\frac{a_{u,d-1}}{b_{u}}\right), \quad p(u)=(a_{u,1},...,a_{u,d-1}, b_{u}) \mbox{ for all }u\in [n],$$ and set $p(0)=(0,\ldots,0)$. Then every $\omega'\in \Stress^a_i(\Delta,p'; \F')$ lifts to an element $\omega\in\Stress^a_i(\Gamma,p;\F)$. Moreover,  $\supp(\omega)=\skel_{i-1}(0*\supp(\omega')).$
	\end{lemma}
	
	In particular, we will frequently use the following corollary.
	\begin{corollary}\label{main cor}
	Let $\Sigma$ be a $(d-1)$-sphere, let $p$ be a generic embedding of $\Sigma$ in $\R^d$, let $1\leq k<d$, and let $i\leq\frac{d-k}{2}$. Then for every $(k-1)$-face $\tau$ of $\Sigma$ with $g_i(\lk(\tau))\geq 1$, there exists $\omega \in \Stress^a_i(\st(\tau), p;\F)$ such that $\supp(\omega)\supseteq \skel_{i-1}(\overline{\tau})$. Here $\F\subseteq \R$ is any field containing $\Q(p(v)_j : v\in V(\Delta), j\in [d])$.
	\end{corollary}
	\begin{proof}
	Throughout the proof, we set $a_{u,j}:=p(u)_j$ for $u\in V(\Sigma), j\in [d]$.

	The proof is by induction on $k$. When $k=1$, $\tau=v$ is a vertex. Consider $$\F'=\Q\left(\frac{a_{u,j}-a_{v,j}}{a_{u,d}-a_{v,d}} : u\in V(\lk(v), j\in[d-1])\right),$$ and define maps $q': V(\lk (v))\to (\F')^{d-1}$ and $q:V(\st (v))\to \F^d$  by 
		\begin{equation} \label{eq:q'-and-q}
			q'(u)=\left(\frac{a_{u,1}-a_{v,1}}{a_{u,d}-a_{v,d}}, \dots, \frac{a_{u,d-1}-a_{v,d-1}}{a_{u,d}-a_{v,d}}\right) \ \forall u\in V(\lk(v)), \;
			q(u)=p(u)-p(v) \; \forall u\in V(\st(v)).
		\end{equation}
		
		Since the set $\{a_{u,i}: u\in \st(v), i\in [d]\}$ is algebraically independent, so is $\{a_{u,i}-a_{v,i} : u\in \lk(v), i\in [d]\}$. 
		Because $g_i(\lk (v)))\geq 1$, there exists a nonzero $\omega' \in\Stress^a_i(\lk (v), q';\F')$.
		By Lemma \ref{cone lemma}, the stress $\omega'$ lifts to a stress $\omega \in \Stress^a_i(\st(v), q;\F)$ with $v$ in its support. 
	Since  translations do not affect the space of affine stresses, we have $\Stress^a_i(\st(v),q;\F)=\Stress^a_i(\st(v),p;\F)$, and the result follows.

	Now assume $k>1$ and $\tau$ is a $(k-1)$-face with $g_i(\lk(\tau))\geq 1$. Let $v$ be a vertex of $\tau$, and set $\tau'=\tau\backslash v$. Then $\lk(v)$ is a $(d-2)$-sphere, $\tau'$ is a face of $\lk(v)$, and $|\tau'|=k-1$. Moreover, $\lk(\tau',\lk(v))=\lk(\tau)$, so $g_i(\lk(\tau',\lk(v)))\geq 1$. Applying the inductive hypothesis  to the face $\tau'$ in the sphere $(\lk(v), q')$, with  $q'$ defined by equation \eqref{eq:q'-and-q}, we obtain a stress $\omega'\in \Stress_i^a(\lk(v), q')$ whose support is contained in $\st(\tau', \lk (v))$  and  contains $\skel_{i-1}(\overline{\tau'})$. Since $\st(\tau,\st(v))=v*\st(\tau',\lk(v))$, Lemma \ref{cone lemma} implies that $\omega'$ lifts to a stress $\omega\in \Stress^a_i(\st(v),q)$ whose support is contained in $\st(\tau)$ and contains $\skel_{i-1}(v*\overline{\tau'})=\skel_{i-1}(\overline{\tau})$. This completes the proof.
	\end{proof}
	
	\begin{remark} \label{rm:sphere->normal-ps}
	The only properties of spheres used in the proof of Corollary \ref{main cor} are: (i) the link of a $(k-1)$-face of a homology $(d-1)$-sphere $\Delta$ is a homology $(d-k-1)$-sphere, and (ii) under a generic embedding $p'$ of $\Sigma$ in $\R^{d'}$,  the dimension of the space $\Stress^a_j(\Delta, p')$ is $g_j(\Delta)$ for all $j\leq d/2$. 

	If $\Delta$ is a normal pseudomanifold, then the link of any face of $\Delta$ is again a normal pseudomanifold. By \cite{Fogelsanger88} and \cite{Tay}, property (ii) continues to hold when $j=2$ and $\Delta$ is a normal pseudomanifold of dimension $\geq 4$. Therefore, Corollary \ref{main cor} also holds when $i=2$ and $\Sigma$ is a normal pseudomanifold. We will use this fact in the proof of Theorem \ref{thm: flag lower bound II}.
	\end{remark}

	We close this section with another application of Lemma \ref{cone lemma}, which will be used in Section 4 and further generalized in Section 5. Recall that the independence number of the graph of a simplicial complex $\Delta$ is denoted by $\alpha(\Delta)$.
	
	\begin{proposition}\label{thm: flag lower bounds-revised}
	Let $i\leq \frac{d-1}{2}$ and $j\geq i+1$, and let $\Delta\in S(d-j, d-1)$. Then $g_i(\Delta)\geq \alpha(\Delta)$. In particular, if $\Delta\in S(\lfloor \frac{d}{2}\rfloor, d-1)$, then $g_i(\Delta)\geq \alpha(\Delta)$ for all $i\leq \frac{d-1}{2}$.
	\end{proposition} 
	
	This result and its proof are motivated by the following recent result of Chudnovsky and Nevo~\cite{Chudnovsky-Nevo}:
	\begin{theorem}\label{thm: flag lower bounds}
		Let $d\geq 5$ and let $\Delta$ be a $(d-1)$-sphere. 
		 If every vertex link of $\Delta$ satisfies $g_2\geq 1$, then $g_2(\Delta)\geq \alpha(\Delta)$.
		In particular, for $d\geq 5$ and $\Delta\in S(1, d-1)$,  $g_2(\Delta)\geq \frac{1-\delta(d)}{2d+1}f_0(\Delta)$, where $\delta(d)\to 0$ as $d\to \infty$.
	\end{theorem} 
	
	\begin{remark} Although \cite{Chudnovsky-Nevo} states the inequality $g_2(\Delta)\geq |I|$ for flag spheres, the proof there relies only on the weaker assumption that every vertex link satisfies $g_2\geq 1$. Consequently, the lower bound on $g_2(\Delta)$ given in the theorem continues to hold for spheres in $S(d-3,d-1)$ and, similarly to Remark \ref{rm:sphere->normal-ps}, even for normal $(d-1)$ pseufomanifolds whose missing faces all have dimension $\leq d-3$. (Here $d\geq 5$.)
	\end{remark}

	 \smallskip\noindent {\it Proof of Proposition \ref{thm: flag lower bounds-revised}: }
	 Consider a generic embedding $p$ of $\Delta$. let $I$ be a maximum independent set in the graph of $\Delta$; that is, $|I|=\alpha(\Delta)$. Since $\Delta\in S(d-j, d-1)$,  all vertex links of $\Delta$ lie in $S(d-j,d-2)$. Because $j\geq i+1$ and $i\leq (d-1)/2$, it follows from Corollary \ref{S(d-i, d-1) implies g_i>0} that $g_i(\lk(v))\ge 1$ for every vertex $v\in \Delta$. Therefore, by Corollary \ref{main cor}, for each $v\in I$ there exists a stress $\omega_v\in \Stress^a_i(\st(v),p) \subseteq \Stress^a_i(\Delta,p)$ whose support contains $v$.
	
	Now let $u,v\in I$ with $u\neq v$. Since $I$ is an independent set, $v\notin \st(u)$, and hence $v$ does not appear in the support of $\omega_u$. Thus each $\omega_v$ has a vertex in its support that does not appear in the support of any other $\omega_u$ with $u\in I$, $u\neq v$. Consequently, the collection $\{\omega_v : v\in I\}$ is linearly independent. Therefore, $g_i=\dim \Stress^a_i(\Delta)\geq |I|=\alpha(\Delta)$. This completes the proof. \hfill$\square$\medskip

	\section{The $g_2$-numbers of spheres in $S(2, 4)$}
	With Lemma \ref{McMullen's formula} and Proposition \ref{thm: flag lower bounds-revised} at our disposal, we are now ready to investigate lower bounds on $g_2$ for spheres in $S(2,4)$. For this discussion, recall that $K(2,4)=\partial\sigma^2*\partial\sigma^2*\partial \sigma^1$ is a sphere in $S(2,4)$. The significance of $K(2,4)$ is that any $\Delta\in S(2,4)$ satisfies $f_0(\Delta)\geq f_0(K(2,4))=8$ \cite[Theorem 3.1]{GoffKleeN};\footnote{In fact, within the class $S(2,4)$, the sphere $K(2,4)$ simultaneously minimizes all face numbers.} moreover,  $K(2,4)$ is the unique $8$-vertex element of $S(2,4)$. 
	
	Also, as mentioned in Section 2.4 (see Corollary \ref{cor: S(k, 2k)}), a sphere $\Delta\in S(2,4)$ satisfies $g_2(\Delta)\geq f_0(\Delta)/4$, while it is conjectured that the much stronger inequality $g_2(\Delta)\geq g_1(\Delta)=f_0(\Delta)-6$ should hold. Meanwhile, since $f_0(\Delta)\geq 8$, for any $\lambda$ with $\frac{1}{4}\leq  \lambda\leq 1$  we have $$\frac{1}{4}f_0(\Delta)\leq  \lambda f_0(\Delta)- (8\lambda-2) \leq f_0(\Delta)-6,$$ and all these inequalities hold as equalities for $\Delta=K(2,4)$. 
	
	The goal of this section is to determine the optimal value of $\lambda$ for which we can prove that  $g_2\geq \lambda f_0- (8\lambda-2)$  for all spheres in $S(2,4)$. We will see in Theorem \ref{thm: S(2,4)} that $\lambda=2/5$ works.
	
	\subsection{Reduction}  \label{sec:reduction}
	
	Given a simplicial complex $\Delta$ and an edge $uv$ that is not contained in any missing face (equivalently, $\lk(uv)=\lk(u)\cap\lk(v)$), one can perform an operation that contracts $uv$. This operation replaces the vertices $u$ and $v$ with a new vertex $u'$, and replaces every face $F$ that intersects $\{u,v\}$ with $(F\backslash\{u,v\})\cup\{u'\}$. The resulting complex $\Delta'$ is still a simplicial complex; moreover, if $\Delta$ is a homology $(d-1)$-sphere, then so is $\Delta'$; see \cite[Proposition 2.3]{NevoNovinsky}. For a sphere in $\Delta \in S(2,4)$ and an edge $e$, we say that contracting $e$ is an {\em admissible contraction} if performing it yields a complex $\Delta'$ that still belongs to $S(2,4)$.
	
	In what follows, we consider a {\em reduced} sphere in $S(2,4)$. We say that $\Delta\in S(2,4)$ is reduced if the following conditions hold:
	\begin{enumerate}
		\item $\Delta$ admits no admissible edge contractions. 
		\item There is no induced subcomplex $\Gamma \cong \partial \sigma^2*\partial \sigma^2\subseteq \Delta$, such that each connected component of $\Delta\backslash\Gamma$ contains at least two vertices. (Here $\Delta\backslash \Gamma$ denotes the induced subcomplex of $\Delta$ on the vertex set complementary to that of $\Gamma$.) 
	\end{enumerate}
	
	To see that it suffices to consider reduced complexes, suppose that $\Delta'\in S(2,4)$ is obtained from $\Delta$ by an admissible contraction of an edge $e$. Then $g_2(\Delta')=g_2(\Delta)-g_1(\lk(e))$. Since $\lk(e)\in S(2, 2)$, we have $g_1(\lk(e))\geq 1$. Recall also that $\lambda\leq 1$. Thus, if  $g_2(\Delta')\geq \lambda f_0(\Delta')-(8\lambda-2)$, then $g_2(\Delta)\geq \lambda f_0(\Delta')-(8\lambda-2)+g_1(\lk(e))\geq \lambda f_0(\Delta)-(8\lambda -2)$, since $f_0(\Delta)=f_0(\Delta')+1$ and $g_1(\lk(e))\geq 1\geq \lambda$. In other words, in this case, $\Delta$ also satisfies $g_2(\Delta)\geq \lambda f_0(\Delta)-(8\lambda-2)$.
	
	Similarly, suppose that $\Delta$ violates condition 2. Let $\Delta'_1$ and $\Delta'_2$ be the two $4$-balls such that $\partial \Delta'_1=\partial \Delta'_2=\Gamma$ and $\Delta'_1 \cup_\Gamma \Delta'_2 =\Delta$. For $i=1,2$, define $\Delta_i=\Delta'_i\cup (\Gamma*v_i)$, where $v_1$ and $v_2$ are new vertices.  Then $\Delta_i\in S(2, 4)$ and $f_0(\Delta_i)<f_0(\Delta)$ for each $i$. Noting that $\Sigma\Gamma=K(2,4)$, and hence $g_2(\Sigma\Gamma)=2=\lambda f_0(\Sigma\Gamma)-(8\lambda-2)$, we conclude that if $g_2(\Delta_i)\geq \lambda f_0(\Delta_i)-(8\lambda -2)$, then 
	\begin{eqnarray*}g_2(\Delta)&=&g_2(\Delta_1)+g_2(\Delta_2)-g_2(\Sigma\Gamma) \\
	&\geq& \left( \sum_{i=1,2}(\lambda f_0(\Delta_i)-(8\lambda-2))\right) -(\lambda f_0(\Sigma\Gamma)-(8\lambda-2))=\lambda f_0(\Delta)-(8\lambda-2).
	\end{eqnarray*}
	Thus, once again $\Delta$  satisfies $g_2(\Delta)\geq \lambda f_0(\Delta)-(8\lambda-2)$.

	This discussion implies the following result:

	\begin{proposition} \label{Prop:reduction}
	If, for some constant $1/4 \leq \lambda \leq 1$, all reduced spheres in $S(2,4)$ satisfy $g_2\geq \lambda f_0-(8\lambda -2)$, then all spheres in $S(2,4)$ satisfy this inequality.
\end{proposition}
	
	The best scenario we can now hope for is the following: a reduced $\Delta\in S(2, 4)$ is either $K(2,4)$, or it admits no admissible edge contraction and, consequently, contains many missing $2$-faces. In the former case, we have $g_2=g_1$. In the latter case, one expects  $g_2$ to be strictly larger than the proposed lower bound. This situation may be compared with the behavior of the $\gamma_2$-numbers of flag $3$-spheres: Venturello \cite{Venturello} constructed a flag $3$-sphere with $12$ vertices such that 1) it is not a suspension, 2) it admits no admissible edge contractions, and 3) except for the vertex links, it has no induced subcomplexes that are $2$-spheres. The $\gamma_2$-number of this sphere is not zero; in fact, it equals $1$.

	\subsection{A new lower bound on $g_2$}
	Throughout, we assume that $\Delta\in S(2, 4)$ is reduced and that $f_0(\Delta)>8$. Our strategy is to derive a lower bound on the number of vertices of $\Delta$ whose links satisfy $g_2\geq 2$, and then use this bound, together with McMullen's formula, to establish a new lower bound on $g_2(\Delta)$. Specifically, our first goal is to prove the following.
	\begin{proposition}  \label{prop:independent}
		Let $v\in \Delta$ be a vertex with $g_2(\lk(v))=1$. Then every neighbor $u$ of $v$ satisfies $g_2(\lk(u))\geq 2$. In particular, the set of vertices $v$ with $g_2(\lk(v))=1$ forms an independent set in the graph of $\Delta$.
	\end{proposition}
	
	In what follows, let $v\in V(\Delta)$ be a vertex with $g_2(\lk(v))=1$. Since $\lk(v)\in S(2, 3)$, Theorem~\ref{thm:Nevo-Novinsky} implies that $\lk(v)$ is the join of two cycles  $C=(u_1, u_2, \dots, u_n, u_1)$ and $C'=(w_1, w_2, w_3, w_1)$, where $n\geq 3$.  We use this notation in the next three lemmas.
	
	\begin{lemma}\label{lm1}
		If $n>3$ and $g_2(\lk(u_2))=1$, then the edge $vu_2$ is contained in a missing $2$-face of the form $vu_2u_j$. Furthermore, the complex $C'*(v, u_2, u_j,v)$ is an induced subcomplex of $\Delta$.
	\end{lemma}
	\begin{proof}
	Since $\Delta\in S(2,4)$, if $e=vu_2$ is not contained in any missing $2$-face, then $e$ is not contained in any missing face at all. Hence, we may contract $e$ to a new vertex $v'$. Denote the resulting complex by $\Delta'$. Then $\Delta'$ is a $4$-sphere. However, by our assumption that $\Delta$ is reduced, $\Delta'$ does not belong to $S(2,4)$, and therefore, it must have a missing $3$-face $F$.
	This face must contain the new vertex $v'$; in other words, $F$ must be of the form $\{v',x,y,z\}$. For this to occur, the subcomplex of $\Delta$ induced by the vertex set $\{v,u_2,x,y,z\}$ must be a $2$-sphere $S\cong \partial\sigma^2*\partial\sigma^1$, and moreover the edge $vu_2$ must belong to $S$. Since the only missing $2$-face in the link of $v$ is $C'$, and since $vu_2$ is an edge of $S$, it follows that $v$ cannot be one of the suspension vertices of $S$. At the same time, since $vu_2$ is not contained in any missing $2$-face, $u_2$ must be one of the suspension vertices of $S$. This implies that $n=4$ and $\lk(v,S)=C=(u_1,u_2,u_3,u_4,u_1)$. 
		
	On the other hand, since $g_2(\lk(u_2))=1$, the complex $\lk(u_2)$ is the join of $C'$ with another cycle containing the path $(u_1, v, u_3)$. The fact that $\partial \overline{u_1vu_3}\subseteq S\subseteq \Delta$ implies that $\lk(u_2)=C'*(u_1, v, u_3, u_1)$. Similarly, $\lk(u_4)$ contains the subcomplex $C'*(u_1, v, u_3)$ as well as the edge $u_1u_3$. Then, for $1\leq i\leq 3$, we have $w_iu_1u_3\in \lk(u_2)$, $w_iu_1u_4, w_iu_3u_4 \in \lk(v)$, and $u_1u_3u_4\in S$. Hence $\partial (w_iu_1u_3u_4)\subseteq \Delta$. Since $\Delta$ has no missing $3$-faces, it follows that $w_iu_1u_3u_4\in \Delta$ for each $w_i$. Thus, all $2$-faces of $C'*(u_1,v,u_3,u_1)$ belong to $\lk(u_4)$. Because $\Delta\in S(2,4)$ and $\lk(u_4)$ is a sphere, we conclude that $\lk(u_4)=C'*(u_1,v,u_3,u_1)$. Consequently, $\Delta\supseteq \st(u_2)\cup \st(u_4)=C'*S$, which implies that $\Delta=C'*S$. This contradicts our assumption that $f_0(\Delta)>8$. Therefore, the edge $e=vu_2$ must be contained in a missing $2$-face. Since $\lk(v)=C*C'$, this missing face must necessarily be of the form $vu_2u_j$.
		
	For the ``furthermore" part, observe first that $w_1w_2w_3$ is a missing $2$-face (otherwise $vw_1w_2w_3$ would be a missing $3$-face of $\Delta$). Also note that each of $\lk(vu_2)$, $\lk(vu_j)$, and $\lk(u_2u_j)$ contains $C'=(w_1,w_2,w_3,w_1)$, and so $C'*(v, u_2, u_j, v)$ is a {\em subcomplex} of $\Delta$. (To justify that $\lk(u_2u_j)$ contains $C'$, note that $\lk(u_2)$ is the join of a $3$-cycle and some other cycle. Since, $\lk(u_2)$ contains $C'$ and since $w_1w_2w_3$ is a missing face, we must have $\lk(u_2)=C'*D$ for some cycle $D$.  Because $u_j\in\lk(u_2)$ but  $u_j\notin C'$, it follows that $u_j\in D$, and hence $\lk(u_2u_j)$ contains $C'$.) Now, both $vu_2u_j$ and $w_1w_2w_3$ are missing $2$-faces. Therefore,  $C'*(v, u_2, u_j, v)$ is an {\em induced} subcomplex of $\Delta$.
	\end{proof}
	
	\begin{lemma}\label{lm2}
		If $\lk(v)$ is the join of two $3$-cycles, then every neighbor $v'$ of $v$ satisfies $g_2(\lk(v'))\geq 2$.
	\end{lemma}
	\begin{proof}
		Assume that $g_2(\lk(u_2))=1$. Then $\lk(u_2)$ is the join of $C'$ with another cycle $C''$ containing the path $(u_1, v, u_3)$. Since $u_1u_2u_3v$ is not a missing $3$-face, it follows that $u_1u_2u_3$ must be a missing $2$-face. Consequently, the cycle $C''$ has length at least $4$. Now interchange the roles of $v$ and $u_2$ and apply Lemma \ref{lm1}. It follows that the vertex $v\in \lk(u_2)$ must be adjacent to a vertex in $\lk(u_2)$ that lies outside the set $\{u_1, u_3\} \cup V(C')$. However, by assumption, $V(\lk(v))=\{u_1, w_1, u_2,w_2, u_3, w_3\}$, and hence $v$ has no such additional neighbor. This contradiction completes the proof.
	\end{proof}
	
	\begin{lemma}\label{lm3}
		If $n>3$, then every neighbor $v'$ of $v$ satisfies $g_2(\lk(v'))\geq 2$.
	\end{lemma}
	\begin{proof}
		We first consider neighbors $v'\in C$. Suppose, for contradiction, that $g_2(\lk(u_2))=1$. By Lemma \ref{lm1}, there exists a missing $2$-face of the form $vu_2u_j$ and $\Gamma=C'*(v, u_2, u_j, v)$ is an induced subcomplex of $\Delta$. Write $\Delta=B_1\cup B_2$, where $B_1, B_2$ are the two $4$-balls whose common boundary is $\Gamma$. We claim that each $B_i$ must contain at least two interior vertices. Otherwise, suppose $B_1=t*\Gamma$. Then $g_2(\lk(t))=1$, and by Lemma \ref{lm2}, the neighbor $u_2$ of $t$ must satisfy $g_2(\lk(u_2))\geq 2$, a contradiction. The above argument shows that $\Delta$ is not reduced, which is again a contradiction. Therefore, each $u_i$ satisfies $g_2(\lk(u_i))\geq 2$.
		
		Next, consider neighbors $w_i \in C'$. The link $\lk(w_i)$ contains the join of $C$ and the path $(w_{i+1}, v, w_{i-1})$. Since $\Delta\in S(2, 4)$, the triangle $w_1w_2w_3$ is a missing $2$-face, so $\lk(w_i)$ cannot be the join of $C$ and a $3$-cycle. Consequently, $g_2(\lk(w_i))\geq 2$. This completes the proof.
	\end{proof}
	
	 \smallskip\noindent {\it Proof of Proposition \ref{prop:independent}: } The statement follows from Lemmas \ref{lm2} and \ref{lm3}. \hfill$\square$\medskip 
	
	One consequence of this discussion is the following result.
	
	\begin{proposition} \label{Prop:reduced}
		If $\Delta\in S(2,4)$ is reduced and $f_0(\Delta)>8$, then  $g_2(\Delta)\geq \frac{2}{5}f_0(\Delta)$.
	\end{proposition}
	\begin{proof}
		Let $I$ denote the set of  vertices $v\in\Delta$ with $g_2(\lk (v))=1$.  By Proposition \ref{prop:independent}, $I$ forms an independent set. There are two cases to consider: either $|I|\geq 2f_0(\Delta)/5$ or $|I|\leq 2f_0(\Delta)/5$. In the former case, $g_2(\Delta)\geq |I|\geq 2f_0(\Delta)/5$ by Theorem \ref{thm: flag lower bounds}. In the latter case, $g_2(\Delta)\geq (2f_0(\Delta)/5+2\cdot 3f_0(\Delta)/5)/4=2f_0(\Delta)/5$ by McMullen's integral formula. In both cases, we conclude that $g_2(\Delta)\geq 2f_0(\Delta)/5$.
	\end{proof}
	
	The main result of this section now follows immediately from Propositions \ref{Prop:reduction} and \ref{Prop:reduced}.

	\begin{theorem}\label{thm: S(2,4)}
		 A sphere in $S(2,4)$ satisfies $g_2\geq \frac{2}{5}f_0-\frac{6}{5}= \frac{2}{5}g_1+\frac{6}{5}g_0$.
	\end{theorem}
	
	\section{The $g$-numbers of spheres in $S(j, d-1)$}   \label{sec:g-numbers}
	In this section, we study the $g_i$-numbers of spheres in $S(j, d-1)$ for general $i$ and $j$. Our approach  is inspired by a result of Chudnovsky and Nevo---Theorem \ref{thm: flag lower bounds}---and its generalization in Proposition \ref{thm: flag lower bounds-revised}. In those results, one uses an independent set in the graph of $\Delta$, together with Corollary \ref{main cor}, to construct a linearly independent collection of stresses on $\Delta$. The main theorems of this section are obtained by further developing and applying this idea to spheres in $S(j, d-1)$.

	\begin{theorem}\label{thm: flag lower bound II}
		Let $d\geq 5$ and let $\Delta$ be a flag normal $(d-1)$-pseudomanifold. 
		Then the following holds:
		\begin{enumerate}
			\item $g_2(\Delta)\geq (d-4)\alpha(\Delta)$. This, in turn, implies that $$g_2(\Delta) \geq \frac{\sqrt{4d^2+12d-31}-(2d+1)}{4} f_0(\Delta).$$
			\item Furthermore, if $\Delta\in S(1, d-1)$, then $g_i(\Delta)\geq (d-2i)\alpha(\Delta)$ for all $2\leq i\leq \frac{d-1}{2}$.		
		\end{enumerate}
	\end{theorem}
	\begin{proof} Let $i\leq \frac{d-1}{2}$, and let
		$I$ be a maximum independent set in the graph of $\Delta$; that is, $|I|=\alpha(\Delta)$. Let $v\in I$ and let $F^v_1:=v$. Since $\Delta$ is flag, the link of any face $F$ is also flag. Consequently, if $\lk(F)$ is nonempty, then it contains two nonadjacent vertices. Specifically, choose nonadjacent vertices $u_2,w_2$  in $\lk(v)$, and let $F^v_2=vu_2$, $G^v_2=vw_2$ be the corresponding $1$-faces of $\Delta$. Continuing inductively, for $2\leq k<d-2i-1$, assume that two $(k-1)$-faces $F^v_k=vu_2\dots u_{k-1}u_k\in\Delta$ and $G^v_k=vu_2\dots u_{k-1}w_k\in\Delta$ have already been defined. Then we choose a pair of nonadjacent vertices $u_{k+1}, w_{k+1}\in\lk(F^v_k)$, and define the corresponding $k$-faces  $F^v_{k+1}=vu_2u_3\dots u_ku_{k+1}$ and $G^v_{k+1}=vu_2u_3\dots u_{k}w_{k+1}$. Finally, we define $G^v_{1}=vu_2u_3\dots u_{d-2i}$. 
		
		First assume that $\Delta\in S(1, d-2)$. Choose a generic embedding of $V(\Delta)$ and, for $2\leq k\leq d-2i$, apply Corollary \ref{main cor} to the face $G^v_k$. Note that $\lk(G_k^v)$ is a flag sphere of dimension $d-k-1$ and that, by our choice of vertices $u_j,w_j$, the edge $vw_k$ is contained in $G_k^v$, while for all $\ell>k$, the edge $vw_\ell$ is not contained in $\st(G_k^v)$. It follows from Corollary \ref{main cor} that, for all $i\leq (d-k)/2$, there exists  an affine $i$-stress $\omega_k^v$ on $\st(G_k^v)$ such that $vw_k\in\supp(\omega_k^v)$, while $vw_\ell\notin \supp(\omega_k^v)$  for all $\ell>k$.
			Likewise, there exists an affine $i$-stress $\omega^v_1$ on $\st(G^v_1)$ such that none of the edges $vw_2, \dots, vw_{d-2i}$ lie in the support of $\omega^v_1$. Hence $\Omega^v:=\{\omega^v_1, \dots, \omega^v_{d-2i}\}$ is a set of linearly independent stresses.
		
		Now consider $\bigcup_{v\in I} \Omega^v$. The support of any stress in $\Omega^v$ contains $v$, but  contains no other vertices of $I\backslash v$. Hence, the set of stresses $\bigcup_{v\in I}\Omega^v$ is also linearly independent. We conclude that 
		\begin{equation} \label{g_j vs (d-2i)I}
		g_i(\Delta)\geq \big|\bigcup_{v\in I}\Omega^v\big|=(d-2i)|I|=(d-2i)\alpha(\Delta).\end{equation}
		
	 	We now consider the case in which $\Delta$ is a flag normal pseudomanifold. In this setting, the link of every face of codimension $\geq 4$ is itself a flag normal pseudomanifold with $g_2\geq 1$ \cite{Tay}. Together with Remark \ref{rm:sphere->normal-ps}, this shows that when $i=2$, the preceding  argument applies equally well to flag normal pseudomanifolds. Consequently, the inequality $g_2(\Delta)\geq (d-4)\alpha(\Delta)$ continues to hold in this broader setting.
		
		To prove the inequality on $g_2$ in terms of $f_0$, let $\epsilon=\frac{\sqrt{4d^2+12d-31}-(2d+1)}{4}$ and assume, for the sake of contradiction, that $g_2< \epsilon f_0$. Then $f_1=g_2+df_0-{d+1 \choose 2}<(d+\epsilon)f_0$.  Hence, by Tur\'an's theorem, there exists an independent set $I$ with $|I| \geq \frac{f_0}{2(d+\epsilon)+1}$. Combining this with \eqref{g_j vs (d-2i)I}, yields $$\epsilon f_0>g_2\geq \frac{(d-4)f_0}{2(d+\epsilon)+1},$$ which implies that $\epsilon> \frac{\sqrt{4d^2+12d-31}-(2d+1)}{4}$. This is the desired contradiction.
	\end{proof}
	
	\begin{remark}
		Since $4d^2+12d-31=(2d+3)^2-22$, it follows that flag normal $(d-1)$-pseudomanifolds satisfy $g_2\geq (1/2-\delta(d))f_0$, where $\delta(d)$ is a function of $d$ with $\delta(d)\to 0$ as $d\to \infty$.
	\end{remark}
	
	Our next result relaxes the flagness assumption of Theorem \ref{thm: flag lower bound II}(2).
	
	\begin{theorem}\label{thm: lower bound on g II}
	Let $\Delta\in S(j, d-1)$, where $j=\min \{\lfloor \frac{d-1}{2}\rfloor-1, d-2i\}$. Then 
	$g_i(\Delta)\geq 2\alpha(\Delta)$ for all $2\leq i\leq d/3$.
	\end{theorem}
	\begin{proof}
		Consider a generic embedding of $\Delta$ in $\R^d$. Let $I$ be an independent set of the graph of $\Delta$ of size $\alpha(\Delta)$, and fix $v\in I$. There are two possible cases to consider.
		
		{\bf Case 1:} The link $\lk(v)$ contains a missing $(r-1)$-face $F=u_1u_2\dots u_r$ for some $2\leq r\leq i$. Define the two $(r-1)$-faces $G_1=vu_1\dots u_{r-2}u_{r-1}$ and $G_2=vu_1\dots u_{r-2}u_{r}$. Since $F$ is a missing face, $G_1\notin\st(G_2)$ and $G_2\notin\st(G_1)$. Next, observe that $i\leq \frac{d-i}{2}\leq \frac{d-r}{2}$. Because $\Delta$ has no missing faces of dimension $\geq d-2i+1$, it also has no missing faces of dimension $\ge d-r-i+1 \ge d-2i+1$. Therefore, $g_i(\lk(G_1))\geq 1$ and $g_i(\lk(G_2))\geq 1$. By Corollary \ref{main cor}, for each $k=1,2$, there exists an affine $i$-stress $\omega^k_v$ on $\st(G_k)$ whose support contains the face $G_k$. Since $G_1\not\in\st(G_2)$ and $G_2\notin\st(G_1)$, we conclude that $\omega^1_v$ and $\omega^2_v$ are linearly independent.
	
		{\bf Case 2:} The link $\lk(v)$ has the complete $(i-1)$-skeleton.  Since $\lk(v)\in S(j, d-2)$ with $j\leq \lfloor \frac{d-1}{2}\rfloor-1$, it follows that $\lk(v)$ has at least $d+2$ vertices. (Indeed, every $(d-2)$-sphere with exactly $d+1$ vertices is of the form $\partial\sigma^\ell*\partial\sigma^{d-1-\ell}$ for some $1\leq\ell\leq d-2$, and therefore has a missing face of dimension $\geq (d-1)/2$.)  Let $W=\{w_1,\dots,w_{d+2}\}$ be any set of $d+2$ such vertices, and define $W_1=W\backslash w_1$ and $W_2=W\backslash w_2$. By our choice of generic embedding, for $k=1,2$, the affine dependence among the points $W_k\cup v$ yields an affine $1$-stress $\delta_k=\sum_{u\in W_k\cup v } \nu_{u, k}x_u\in \Stress^a_1(\Delta)$, where $\nu_{u, k}\neq 0$ for all $u\in W_k\cup v$. Since $\skel_{i-1}(\overline{W})*v$ is a subcomplex of $\Delta$, it follows that $\omega_v^1:=(\delta_1)^i$ and $\omega_v^2:=(\delta_2)^i$ are affine $i$-stresses with supports $\skel_{i-1}(\overline{W_1\cup v})$ and $\skel_{i-1}(\overline{W_2\cup v})$, respectively. These supports are distinct, and therefore $\omega_v^1$ and $\omega_v^2$ are linearly independent.
		
		Finally, since $I$ is an independent set, the collection $\bigcup_{v\in I} \{\omega^1_v, \omega^2_v\}$ of affine $i$-stresses is linearly independent. This gives $g_i(\Delta)\geq 2\alpha(\Delta)$, as claimed.
	\end{proof}
	
Theorem \ref{thm: flag lower bound II}(1) provides a lower bound on  $g_2$ in terms of $f_0$ for flag spheres. The next theorem generalizes this result, giving a lower bound on $g_{k+1}$ in terms of $f_{k-1}$ for certain values of $k>1$, for spheres without large missing faces.
	
	\begin{theorem}\label{thm: g_{k+1} in terms of f_{k-1}}
	Let $k\geq 2$ and $d\geq 3(k+1)$. Then there exists a positive constant $c_k$, depending only on $k$, such that for every sphere $\Delta\in S(d-2-2k,d-1)$ with sufficiently many vertices,
	\begin{enumerate}
	\item $g_{k+1}(\Delta)\geq c_k\cdot \big(f_{k-1}(\Delta)\big)^{\frac{2k+2}{3k+1}}$ if $k\leq d/4$, and
	\item $g_{k+1}(\Delta)\geq c_k\cdot \big(f_{k-1}(\Delta)\big)^{\frac{2k+2}{k+1+\lfloor d/2\rfloor}}$ if $k\geq d/4$.
	\end{enumerate}
	\end{theorem}
	\begin{proof}
		We first treat the case $k\leq d/4$. Define the following graph $K$, which has $f_{k-1}(\Delta)$ vertices: each vertex corresponds to a $(k-1)$-face of $\Delta$, and two vertices corresponding to $(k-1)$-faces $\sigma$ and $\tau$  are connected by an edge if $\sigma\cup\tau$ is a face of $\Delta$. In this case, $\dim(\sigma\cup\tau)=2k-i-1$ for some $0\leq i\leq k-1$. Since every set of size $2k-i$ can be expressed as the union of two of its $k$-subsets in $\frac{1}{2}\binom{2k-i}{k}\binom{k}{i}$ ways, it follows that the number of edges of $K$ satisfies
		$$f_1(K)=\sum_{i=0}^{k-1} \frac{1}{2}\binom{2k-i}{k}\binom{k}{i}f_{2k-i-1}(\Delta) < \frac{1}{2}\binom{2k}{k}\binom{k}{\lfloor k/2\rfloor}\sum_{i=0}^{k-1} f_{2k-i-1}(\Delta).$$
		We conclude that the average vertex degree of $K$ is $\frac{2f_1(K)}{f_0(K)}<\frac{\binom{2k}{k}\binom{k}{\lfloor k/2\rfloor}\sum_{i=0}^{k-1} f_{2k-i-1}(\Delta)}{f_{k-1}(\Delta)}$. Hence, by T\'uran's theorem, the independence number $\alpha(K)$ of $K$ satisfies
		$$\alpha(K)> \frac{\big(f_{k-1}(\Delta)\big)^2}{\binom{2k}{k}\binom{k}{\lfloor k/2\rfloor}\sum_{i=0}^{k} f_{2k-i-1}(\Delta)}.$$
		
		Consider a generic embedding of $\Delta$ in $\R^d$. If $\tau\in \Delta$ is a $(k-1)$-face, then $\lk(\tau)\in S(d-2-2k, d-1-k)$. In particular, $g_{k+1}(\lk(\tau))\geq 1$. Let $m=\alpha(K)$ and let $\tau_1, \tau_2, \dots, \tau_m$ be $(k-1)$-faces of $\Delta$ representing a maximum independent set of $K$. Then for any $1\leq i\neq j\leq m$, we have $\tau_i\notin \st(\tau_j)$ because $\tau_i\cup \tau_j \notin\Delta$. Hence, by Corollary~\ref{main cor}, for each $1\leq i\leq m$, there exists an affine $(k+1)$-stress $\omega_i\in \Stress^a_{k+1}(\st(\tau_i))$ such that $\tau_i\in \supp(\omega_i)$, but $\tau_j\notin \supp(\omega_i)$ for all $j\neq i$. Consequently, the stresses $\omega_1, \dots, \omega_m$ are linearly independent. We conclude that $$g_{k+1}(\Delta)\geq m =\alpha(K)> \frac{\big(f_{k-1}(\Delta)\big)^2}{\binom{2k}{k}\binom{k}{\lfloor k/2\rfloor}\sum_{i=0}^{k} f_{2k-i-1}(\Delta)}. $$

		For the remainder of the proof, all the $g$- and $f$-numbers refer to those of $\Delta$. Write $g_{k+1}=\binom{x}{k+1}$ for some real number $x\geq k+1$. Assume that $g_{k+1}<c_k \cdot f_{k-1}^{\frac{2k+2}{3k+1}}$, where $c_k$ is a positive constant to be determined later. Then $$c_k\cdot f_{k-1}^{\frac{2k+2}{3k+1}}> g_{k+1}=\binom{x}{k+1}>\frac{(x-k)^{k+1}}{(k+1)!}.$$
		That is, $x-k<((k+1)!c_k)^{\frac{1}{k+1}}\cdot \big(f_{k-1}(\Delta)\big)^{\frac{2}{3k+1}}.$ Using the $g$-theorem, along with the fact that for $1\leq s\leq k$ and $x\geq k+1$, $x+s-1\leq(k+s)(x-k)$, we infer that for all $1\leq s\leq k$,
		$$g_{k+s}\leq \binom{x+s-1}{k+s}<\frac{(x+s-1)^{k+s}}{(k+s)!}\leq \frac{(k+s)^{k+s}(x-k)^{k+s}}{(k+s)!}\leq c_{k,s}\cdot f_{k-1}^{\frac{2(k+s)}{3k+1}},$$
		where $c_{k,s}=\frac{(k+s)^{k+s}}{(k+s)!}((k+1)!c_k)^{\frac{k+s}{k+1}}$.
		
		Since $\Delta\in S(d-2-2k,d-1)$, we have $g_{2k}\geq 1$. From the defining relations for the $g$-numbers, each $f_{2k-i-1}$ can be expressed as a linear combination of $g_{2k}, g_{2k-1}, \dots, g_{k+2}, g_{k+1}, f_{k-1}, \dots, f_0, 1$. Moreover, when $i=0$, the coefficient of $g_{2k}$ in this linear combination is $1$, while for $i>0$ this coefficient is zero.  By Bj\"orner's result \cite[Theorem 5]{Bjorner-93}, the first half of the $f$-numbers of $\Delta$ is weakly increasing. Thus, for $k\leq d/4$, $f_0\leq f_1\leq \dots \leq f_{k-1}$. Also the exponents $\{2(k+s)/(3k+1)\}_{s=1}^k$ form an increasing sequence, whose largest term $4k/(3k+1)>1$. Therefore, letting $c'_k=c_{k,k}$, we obtain that for $f_0\gg 0$,
		$$\sum_{i=0}^{k} f_{2k-i-1}< 2c'_k\cdot f_{k-1}^{\frac{4k}{3k+1}}.$$
		
		Putting everything together, we conclude that
		$$c_k\cdot f_{k-1}^{\frac{2k+2}{3k+1}}>g_{k+1}> \frac{f_{k-1}^2}{\binom{2k}{k}\binom{k}{\lfloor k/2\rfloor}\sum_{i=0}^{k} f_{2k-i-1}}
		>\frac{f_{k-1}^{\frac{2k+2}{3k+1}}}{2\binom{2k}{k}\binom{k}{\lfloor k/2\rfloor}\cdot c_k'}.$$ Choosing $c_k>0$ so that  $c_k\cdot c_k'\cdot 2\binom{2k}{k}\binom{k}{\lfloor k/2\rfloor}<1$ leads to a contradiction. Thus, for such a choice of $c_k$, $g_{k+1}\geq c_k\cdot f_{k-1}^{\frac{2k+2}{3k+1}}$.
		
		The treatment of the case $\frac{d}{4}\leq k \leq \frac{d}{3}-1$ is very similar. The only difference is that when $k\geq d/4$, each of $f_k,f_{k+1},\dots,f_{2k-1}$ can be written as a linear combination of $$g_{\lfloor d/2\rfloor}, \dots, g_{k+2},g_{k+1}, f_{k-1}, \dots, f_0, 1.$$ Since the inequalities $g_{k+s} \leq O\left((g_{k+1})^{(k+s)/(k+1)}\right)$ continue to hold for all $1\leq s\leq \lfloor d/2\rfloor-k$, and since $g_{\lfloor d/2\rfloor}>0$ for $\Delta\in S(d-2-2k,d-1)$, computations analogous to those in the previous case apply. These yield the bound stated in the theorem.
	\end{proof}

	\section{Level rings and counterexamples to conjectures on stresses}
	In this section, we switch gears and use algebraic tools to obtain additional relations among the $g$-numbers of spheres in $S(j, d-1)$.
	
	We begin by reviewing some basics on Gorenstein and level rings. Let $R=\R[x_1,\ldots,x_n]$, let $\mathbf{m}=(x_1, \dots, x_n)$ be the irrelevant ideal (that is, the maximal homogeneous ideal) of $R$, and let $N$ be a graded $R$-module. One example of such $N$ is $\R[\Delta]=R/I_\Delta$. In fact, all of the $R$-modules we consider  are quotients of $\R[\Delta]$.
	
	Define the {\em socle} of $N$ by $\Soc^R N=0:_N \mathbf{m}=\{y\in N: \mathbf{m}y=0\}$. If $N$ has Krull dimension $0$, then the largest $k$ for which $r_k(N)\neq 0$ is called the {\em socle degree of $N$}; we denote it by $s$. The integer vector $S(N)=(r_1(N), \dots, r_s(N))$ is called the {\em socle vector of $N$}. 
	
	If $N=\R\oplus N_1 \oplus \dots \oplus N_s$ (with $N_s\neq 0$) as an $\R$-vector space, then $(\Soc^R N)_s=N_s$, and hence the socle degree of $N$ is $s$. We say that $N$ is {\em level} if its socle vector $S(N)$ is of the form $(0,\dots, 0, a)$ for some $a\geq 1$, and {\em Gorenstein} if in addition $a=1$. A sequence is called a {\em level sequence} (resp.~\emph{Gorenstein sequence}) if it is the the Hilbert function of an Artinian level ring (resp.~Artinian Gorenstein ring) $N$; that is, it is of the form $(1, \dim_\R N_1, \dots, \dim_\R N_s)$. In particular, every level sequence is an $M$-sequence.
	
	Assume that $(\Delta, p)$ is a $(d-1)$-sphere with $n$ vertices and with an embedding $p: V(\Delta)\to \R^d$ that is either generic or natural (when $\Delta$ is polytopal). As in Section 2, let $\Theta=\Theta(p)$ be the l.s.o.p.~of $\R[\Delta]$ associated with $p$ and let $c=\sum_{v=1}^n x_v$. Recall that $\R[\Delta]/(\Theta)$ is Gorenstein. On the other hand, if we set $A(\Delta):= \R[\Delta]/(\Theta,c)$, then by the $g$-theorem, we have $\dim_\R A(\Delta)_k=g_k$ for $0\leq k\leq \lfloor d/2\rfloor$ and $\dim_\R A(\Delta)_k=0$  otherwise. So the socle degree of $A(\Delta)$ is at most $\lfloor d/2\rfloor$.

The relevance of level rings to spheres with no large missing faces is explained by the following result; see \cite[Proposition 3.2]{MNZ}.

\begin{proposition} \label{MNZ-Prop3.2}
Let $(\Delta, p)$ be a $(d-1)$-sphere with $n$ vertices and with a generic or natural embedding $p$ in $\R^d$. Let $\Theta=\Theta(p)$ be the l.s.o.p.~associated with $p$. Denote by $m_i$ the number of missing $i$-faces of $\Delta$. Then 
\begin{eqnarray} \nonumber \dim_\R \Soc^R A(\Delta)_k &=& m_{d-k} \quad\mbox{if } k<\lfloor(d-1)/2\rfloor, \mbox{ and}\\
	\label{eq:mid-dim}		\dim _\R\Soc^R A(\Delta)_k &\geq& m_{d-k} \quad\mbox{if } k=\lfloor(d-1)/2\rfloor.
\end{eqnarray}		
\end{proposition} 	

An immediate, albeit rather surprising, consequence of this result and the definition of level rings is the following.

\begin{corollary}\label{cor: level ring}
Let $\Delta$ be a $(d-1)$-sphere, 
let $u$ be the largest integer such that $\Delta \in S(d-u, d-1)$, and set
 $\tilde{u}:=\min\{u, \lfloor(d-1)/2\rfloor\}$. Then, under the assumptions and notation of Proposition \ref{MNZ-Prop3.2}, the graded algebra $\oplus_{i=0}^{\tilde{u}} A(\Delta)_i$ is a level ring of socle degree $\tilde{u}$. In particular, $(g_0(\Delta),g_1(\Delta), \dots, g_{\tilde{u}}(\Delta))$ is a level sequence.
\end{corollary}
\begin{proof}
		Since $\Delta \in S(d-u, d-1)$, Proposition \ref{MNZ-Prop3.2} implies that $A(\Delta)$ has vanishing socle in all degrees $<\tilde{u}$. Therefore, the truncated algebra $\oplus_{i=0}^{\tilde{u}} A(\Delta)_i$ is a level ring of socle degree $\tilde{u}$. 
\end{proof}

Level sequences have been extensively studied, and, as the next two theorems demonstrate, levelness imposes far stronger restrictions on the entries of a sequence than merely being an $M$-sequence. The following theorem summarizes several results from \cite[Theorem 2.2]{GHMS} and \cite[Section III.3]{Stanley96}.
	\begin{theorem}\label{thm: level sequence}
		Let $\ell=(1, \ell_1, \dots, \ell_s)$ be a level sequence. Then:
		\begin{enumerate}
			\item For every $1\leq s' \leq s$, the truncated sequence $(1, \ell_1, \dots, \ell_{s'})$ is also level.
			\item For all $i, j \geq 1$ with $i+j\leq s$, one has $\ell_i\leq \ell_j\ell_{i+j}$.
			\item The reverse sequence $(\ell_s, \dots, \ell_1, 1)$ is the sum of a Gorenstein sequence $(1, b_1, \dots, b_{s-1}, 1)$ and of $\ell_s-1$ additional $M$-sequences (each ending in a zero). 
			In particular, if $\ell_s=1$, then $\ell$ is Gorenstein, and hence symmetric.
		\end{enumerate}
	\end{theorem}  
	
	In view of part 3 of this theorem, it is natural to look for structural properties satisfied by sums of $M$-sequences. Several such properties are collected in the following theorem; see \cite[Proposition 4.3]{LS}. This theorem follows from a stronger result \cite{BE,Hulett}, which extends Macaulay's theorem to graded modules.
	
	\begin{theorem}\label{thm: sum of M sequences}
		Let $\ell=(\ell_0, \dots, \ell_s)$ be a sequence of nonnegative integers. The following statements are equivalent:
		\begin{enumerate}
			\item The sequence $\ell$ is a sum of $M$-sequences.
			\item Either $\ell_i=0$ for all $i$, or $\ell_0\geq 1$ and the sequence $(1, \ell_1, \dots, \ell_s)$ is an $M$-sequence.
			\item The sequence $\ell$ is the Hilbert function of a graded module over some polynomial ring $\R[x_1, \dots, x_m]$ generated in degree $0$.
		\end{enumerate}
	\end{theorem}
	
Putting Corollary \ref{cor: level ring} and Theorems \ref{thm: level sequence} and \ref{thm: sum of M sequences} together yields the following restrictions on the $g$-numbers of spheres without large missing faces. These relations constitute the main result of this section.
	
\begin{corollary}\label{thm: level ring}
Let $\Delta$ be a $(d-1)$-sphere, let $u$ be the largest integer such that $\Delta \in S(d-u, d-1)$, and set $\tilde{u}:=\min\{u, \lfloor(d-1)/2\rfloor\}$. Then the sequence $(g_0(\Delta),g_1(\Delta), \dots, g_{\tilde{u}}(\Delta))$ is a level sequence. In particular,
		\begin{enumerate}
			\item For all $i, j \geq 1$ with $i+j\leq \tilde{u}$, one has $g_i\leq g_j g_{i+j}$.
			\item Both $(1,g_1,\dots, g_{\lfloor d/2\rfloor})$ and $(1, g_{\tilde{u}-1}, \dots, g_{1}, 1)$ are $M$-sequences.
		\end{enumerate} 
\end{corollary}
	
The following special case of Corollary \ref{thm: level ring} deserves separate mention:
	\begin{corollary}\label{cor1}
		Let $\Delta \in S(u+1, 2u)$. Then  $g(\Delta)=(g_0(\Delta),g_1(\Delta), \dots, g_{u}(\Delta))$ is a level sequence. In particular, if $g_u(\Delta)=1$, then the $g$-vector of $\Delta$ is a Gorenstein sequence and hence symmetric.
	\end{corollary}
	
	In light of Corollary \ref{cor1}, one may wonder whether the $g$-vectors of spheres in $S(u, 2u-1)$ are also level. To investigate this question, recall the following conjecture on the structure of the spaces of affine stresses, proposed in \cite[Conjecture 3.6]{NZ-Aff-Reconstr}. 
	\begin{conjecture}\label{conj: affine reconstruction}
		Assume $u\leq d/2$. Let $(\Delta, p)$ be a $(d-1)$-sphere in $S(d-u,d-1)$ with a generic or natural embedding $p$ in $\R^d$. Then for all $1\leq j<i\leq u$,
		\begin{equation}\label{eq: u-1-stresses from u-stresses}
			\Stress^a_{j}(\Delta,p)=\{\partial_{\mu}\omega : \omega\in \Stress^a_i(\Delta,p), \mu\in R_{i-j}\}.
		\end{equation}
	\end{conjecture}
	
	This conjecture was proved in \cite[Theorem 3.1]{MNZ} for all $u\leq \lceil\frac{d}{2}\rceil-1$. The case $u=d/2$ was left open. More precisely, when $d=2u$, \cite[Theorem 3.1]{MNZ} established that \eqref{eq: u-1-stresses from u-stresses} holds for all $1\leq j<i\leq u-1$; moreover, it showed that the validity of \eqref{eq: u-1-stresses from u-stresses} for $j=u-1$, $i=u$ is equivalent to the inequality of \eqref{eq:mid-dim} holding as equality. In other words, when $d=2u$, 
	$$\Stress^a_{u-1}(\Delta,p)=\{\partial_{x_v}\omega : \omega\in \Stress^a_u(\Delta,p), v\in[n]\} \quad \mbox{if and only if} \quad \dim_\R \Soc^R A(\Delta)_{u-1} = m_{2u-(u-1)}=0.$$ That is, Conjecture \ref{conj: affine reconstruction} holds for $d=2u$ and $\Delta\in S(u,2u-1)$ if and only if $A(\Delta)$ is a level ring. 
	
	Our next task in this section is to construct, for all $u\geq 3$, a sphere in $S(u,2u-1)$ whose $g$-vector is not level, thereby showing that Conjecture \ref{conj: affine reconstruction} is false when $d=2u\geq 6$. Meanwhile, it is worth noting that Conjecture \ref{conj: affine reconstruction} does hold when $u=2$,  $\Delta\in S(2,3)$, and $p$ is generic; this was proved in \cite[Theorem 8.4]{CJT}. The case where $p$ is natural remains open. We also mention that Conjecture \ref{conj: affine reconstruction} holds when $d=2u$ and $\Delta$ is a flag PL $(d-1)$-sphere; see \cite[Theorem 6.3]{NZ-Aff-Reconstr}. In particular, $A(\Delta)$ is level in these cases.

	\begin{example}\label{Example: support of stress}
		Let $u\geq 3k\geq 3$, and consider the sphere $$K(u-k, 2u-1)=\partial \sigma^{u-k}*\partial \sigma^{u-k}*\partial \sigma^{2k}\in S(u-k, 2u-1).$$ Since $K(u-k, 2u-1)$ is a join of boundaries of simplices, we have
		$$h(K(u-k, 2u-1), t)=h(\partial \sigma^{u-k}, t)^2h(\partial \sigma^{2k},t)=(1+t+\dots+t^{u-k})^2(1+t+\dots+t^{2k}).$$
	When $j\leq 2k$, the coefficient $h_j (K(u-k, 2u-1))$ counts the number of weak compositions of $j$ into three parts. Consequently,  $g_j(K(u-k, 2u-1))=\binom{j+2}{2}-\binom{j+1}{2}=j+1$. Similarly, when $2k\leq  j\leq u-k$, the value $h_j (K(u-k, 2u-1))$ is the sum, over $0\leq \ell\leq 2k$, of the numbers of weak compositions of
	$j - \ell$ into two parts. Thus, in this range, $h_j(K(u-k, 2u-1)) = \sum_{\ell=0}^{2k} (j-\ell+1)$, and hence $g_j(K(u-k, 2u-1)) =\sum_{\ell=0}^{2k} 1=2k+1$ for all $2k+1\leq j\leq u-k$.
		Finally, when $u-k \leq j\leq u$,
		$h_j(K(u-k, 2u-1)) = \sum_{\ell=0}^{2k} h_{j-\ell}(\partial \sigma^{u-k}*\partial \sigma^{u-k})$. Therefore, for $u-k<j\leq u$, $$g_j(K(u-k, 2u-1))=(h_j-h_{j-2k-1})(\partial \sigma^{u-k}*\partial \sigma^{u-k})=(2(u-k)+1-j)-(j-2k)=2u+1-j.$$
		To summarize, $$g_j(K(u-k, 2u-1))=\begin{cases}
			j+1 & \mbox{if }0\leq j\leq 2k\\
			2k+1 & \mbox{if }2k+1\leq j\leq u-k\\
			2u+1-2j & \mbox{if }u-k+1 \leq j\leq u
		\end{cases}.$$
		\end{example}
	
	\begin{corollary}\label{counterexample1}
		Conjecture \ref{conj: affine reconstruction} does not hold when $d=2u\geq 6$: for all $u\geq 3$, there exists a polytopal sphere $\Delta\in S(u,2u-1)$ such that, regardless of whether $p$ is generic or natural, the ring $A(\Delta)$ is not level.
	\end{corollary}
	
	\begin{proof} By the discussion following Conjecture \ref{conj: affine reconstruction}, if $\Delta\in S(u,2u-1)$ satisfies the conjecture, then the $g$-vector of $\Delta$ must be a level sequence. In particular, if $g_u(\Delta)=1$, then $g(\Delta)$ must be Gorenstein, and hence symmetric. Now consider $K(u-k, 2u-1)$ with $k\geq 1$ and $u\geq 3k$. The computations carried out in Example \ref{Example: support of stress} show that while $g_u(K(u-k, 2u-1))=1$, one has $g_1(K(u-k, 2u-1))=2$ and $g_{u-1}(K(u-k, 2u-1))=3$, so the required symmetry fails. Thus, the complexes $K(u-k, 2u-1)\in S(u-k,2u-1)\subseteq S(u,2u-1)$ violate the conjecture.
	\end{proof}
	
	We recall another conjecture---this one is on the support of affine stresses---proposed in \cite[Conjecture 3.3]{NZ-Aff-Reconstr}, which we are now able to disprove.

	\begin{conjecture}\label{conj: support of stresses}
		Let $2\leq i\leq d/2$, let $\Delta\in S(d-i,d-1)$, and let $p$ be a generic or natural embedding of $\Delta$ in $\R^d$. Then every $(i- 1)$-face of $\Delta$ participates in some affine $i$-stress on $(\Delta, p)$. 
	\end{conjecture}
	
	While Conjecture \ref{conj: support of stresses} holds for $i=2$ (see \cite{Z-rigidity} for the case of generic embeddings and \cite[Theorem 1.5]{MNZ} for the case of natural embeddings), for following result shows that it fails in general.
	\begin{corollary}\label{counterexample2}
	For any $m\geq 1$, there exists a simplicial $(4m+2)$-polytope $P$ with $\partial P\in S(2m,4m+1)$, and a $2m$-face of $P$ that does not participate in any affine $(2m+1)$-stress on $(\partial P,p)$, where $p$ is the natural embedding of $\partial P$.
	\end{corollary}
	\begin{proof}
	Let $u=2m+1$ and consider the sphere $K(2m, 4m+1)=\partial \sigma^{2m}*\partial \sigma^{2m}*\partial \sigma^2\in S(2m, 4m+1)$ from Example \ref{Example: support of stress}. This sphere can be realized as the boundary of a simplicial $(4m+2)$-polytope $P$ on vertex set $[2u+3]$, where the position vectors of the vertices $1,2,\dots, 2u+3$ are given by
		$$ e_1,e_2,\dots,e_{u-1}, -\sum_{i=1}^{u-1}e_i,\quad e_u,e_{u+1},\dots,e_{2u-2}, -\sum_{i=u}^{2u-2}e_i, \quad e_{2u-1},e_{2u}, -(e_{2u-1}+e_{2u}). $$
		
		Let $y_1=x_1+x_2+\dots+x_u$, $y_2=x_{u+1}+x_{u+2}+\dots+x_{2u}$, and $y_3=x_{2u+1}+x_{2u+2}+x_{2u+3}$. Then the definition of $p$ implies that $y_1-\frac{u}{3}\cdot y_3$ and  $y_2-\frac{u}{3}\cdot y_3$ belong to $\Stress^a_1(\partial P, p)$. Since the only missing faces of $\partial P$ are $\{1, 2,\dots, u\}$, $\{u+1, u+2,\dots,2u\}$, and $\{2u+1, 2u+2, 2u+3\}$, we conclude that $f=(y_1-\frac{u}{3}\cdot y_3)(y_2-\frac{u}{3}\cdot y_3)(y_1-y_2)^{u-2}$ is an element of $\Stress^a_u(\partial P, p)$. Furthermore, since  $g_u(\partial P)=1$, the stress $f$ forms a basis of $\Stress^a_u(\partial P, p)$. Now, the coefficient of $y_1^my_2^my_3$ in  $f$ is 
		$$-\frac{2m+1}{3}\cdot \binom{2m-1}{m}(-1)^{m-1} -\frac{2m+1}{3}\cdot \binom{2m-1}{m-1}(-1)^m=0.$$ In particular, no $2m$-face of $\partial P$ the form $F\cup G\cup v$, where $F$ is an $m$-subset of $\{1, 2, \dots, u\}$, $G$ is an $m$-subset of $\{u+1, u+2, \dots, 2u\}$, and $v\in \{2u+1, 2u+2, 2u+3\}$, is in the support of $f$. Therefore, no face of this form is in the support of any element of $\Stress^a_u(\partial P, p)$, and so
		the conjecture fails in this case.
	\end{proof}

\	{\small
		\bibliography{refs}

\begin{thebibliography}{10}

\bibitem{Adiprasito-g-conjecture}
K.~Adiprasito.
\newblock Combinatorial {L}efschetz theorems beyond positivity.
\newblock arXiv:1812.10454v4, 2018.

\bibitem{AdiprasitoPapadakisPetrotou}
K.~Adiprasito, S.~A. Papadakis, and V.~Petrotou.
\newblock Anisotropy, biased pairings, and the {L}efschetz property for
  pseudomanifolds and cycles.
\newblock arXiv:2101.07245v2, 2021.

\bibitem{EignerZiegler}
M.~Aigner and G.~M. Ziegler.
\newblock {\em Proofs from {\it {T}he {B}ook}}.
\newblock Springer, Berlin, sixth edition, 2018.

\bibitem{Barnette-LBT-pseudomanifolds}
D.~Barnette.
\newblock Graph theorems for manifolds.
\newblock {\em Israel J. Math.}, 16:62--72, 1973.

\bibitem{Bjorner-93}
A.~Bj\"{o}rner.
\newblock Partial unimodality for {$f$}-vectors of simplicial polytopes and
  spheres.
\newblock In {\em Jerusalem combinatorics '93}, volume 178 of {\em Contemp.
  Math.}, pages 45--54. Amer. Math. Soc., Providence, RI, 1994.

\bibitem{BE}
C.~Blancafort and J.~Elias.
\newblock On the growth of the {H}ilbert function of a module.
\newblock {\em Math. Z.}, 234(3):507--517, 2000.

\bibitem{CharneyDavis95}
R.~Charney and M.~Davis.
\newblock The {E}uler characteristic of a nonpositively curved, piecewise
  {E}uclidean manifold.
\newblock {\em Pacific J.~Math.}, 171:117--137, 1995.

\bibitem{Chudnovsky-Nevo}
M.~Chudnovsky and E.~Nevo.
\newblock Stable sets in flag spheres.
\newblock {\em European J. Combin.}, 110:103699, 2023.

\bibitem{CJT}
J.~Cruickshank, B.~Jackson, and S.~i.~Tanigawa.
\newblock Global rigidity of triangulated manifolds.
\newblock {\em Adv. Math.}, 458(part A):Paper No. 109953, 59, 2024.

\bibitem{DavisOkun}
M.~Davis and B.~Okun.
\newblock Vanishing theorems and conjectures for the $\ell^2$-homology of
  right-angled coxeter groups.
\newblock {\em Geom.~Topol.}, 5:7--74, 2001.

\bibitem{Fogelsanger88}
A.~Fogelsanger.
\newblock {\em The generic rigidity of minimal cycles}.
\newblock PhD thesis, Cornell University, 1988.

\bibitem{Gal05}
S.~Gal.
\newblock Real root conjecture fails for five- and higher-dimensional spheres.
\newblock {\em Discrete Comput.~Geom.}, 34:269--284, 2005.

\bibitem{GHMS}
A.~V. Geramita, T.~Harima, J.~C. Migliore, and Y.~S. Shin.
\newblock The {H}ilbert function of a level algebra.
\newblock {\em Mem. Amer. Math. Soc.}, 186(872):vi+139, 2007.

\bibitem{GoffKleeN}
M.~Goff, S.~Klee, and I.~Novik.
\newblock Balanced complexes and complexes without large missing faces.
\newblock {\em Ark.~Mat.}, 49:335--350, 2011.

\bibitem{Gromov87}
M.~Gromov.
\newblock Hyperbolic groups.
\newblock In S.~M. Gersten, editor, {\em Essays in group theory}, volume~8 of
  {\em Math. Sci. Res. Inst. Publ.}, pages 75--263. Springer, New York, 1987.

\bibitem{Hulett}
H.~A. Hulett.
\newblock A generalization of {M}acaulay’s theorem.
\newblock {\em Comm. Algebra}, 23(4):1249--1263, 1995.

\bibitem{Kalai87}
G.~Kalai.
\newblock Rigidity and the lower bound theorem.{ I}.
\newblock {\em Invent.~Math.}, 88:125--151, 1987.

\bibitem{KaruXiao}
K.~Karu and E.~Xiao.
\newblock On the anisotropy theorem of {P}apadakis and {P}etrotou.
\newblock {\em Algebr. Comb.}, 6(5):1313--1330, 2023.

\bibitem{Klee64}
V.~Klee.
\newblock A combinatorial analogue of {P}oincar\'e's duality theorem.
\newblock {\em Canad. J. Math.}, 16:517--531, 1964.

\bibitem{LS}
M.~Larson and A.~Stapledon.
\newblock Complementary vectors of simplicial complexes.
\newblock arXiv:2504.20264, 2025.

\bibitem{McMullen70}
P.~McMullen.
\newblock The maximum numbers of faces of a convex polytope.
\newblock {\em Mathematika}, 17:179--184, 1970.

\bibitem{McMullen96}
P.~McMullen.
\newblock Weights on polytopes.
\newblock {\em Discrete Comput. Geom.}, 15(4):363--388, 1996.

\bibitem{McMullenWalkup71}
P.~McMullen and D.~W. Walkup.
\newblock A generalized lower-bound conjecture for simplicial polytopes.
\newblock {\em Mathematika}, 18:264--273, 1971.

\bibitem{MuraiNevo2013}
S.~Murai and E.~Nevo.
\newblock On the generalized lower bound conjecture for polytopes and spheres.
\newblock {\em Acta Math.}, 210(1):185--202, 2013.

\bibitem{MNZ}
S.~Murai, I.~Novik, and H.~Zheng.
\newblock Affine stresses, inverse systems, and reconstruction problems.
\newblock {\em Int. Math. Res. Not. IMRN}, pages 8540--8556, 2024.

\bibitem{Nevo2009}
E.~Nevo.
\newblock {Remarks on missing faces and generalized lower bounds on face
  numbers}.
\newblock {\em Electron.~J.~Combin.}, 16(2):Research Paper 8, 2009.

\bibitem{NevoNovinsky}
E.~Nevo and E.~Novinsky.
\newblock A characterization of simplicial polytopes with $g_2 = 1$.
\newblock {\em J. Combin. Theory Ser. A}, 118:387--395, 2011.

\bibitem{NZ-Aff-Reconstr}
I.~Novik and H.~Zheng.
\newblock Affine stresses: the partition of unity and {K}alai's reconstruction
  conjectures.
\newblock {\em Discrete Comp.~Geom.}, 72:928--956, 2024.

\bibitem{NZ-g_i}
I.~Novik and H.~Zheng.
\newblock Simplicial spheres with $g_k=1$.
\newblock arXiv:2601.10072, 2026.

\bibitem{PapadakisPetrotou}
S.~A. Papadakis and V.~Petrotou.
\newblock The characteristic 2 anisotropicity of simplicial spheres.
\newblock arXiv:2012.09815, 2020.

\bibitem{Stanley80}
R.~P. Stanley.
\newblock The number of faces of a simplicial convex polytope.
\newblock {\em Adv.~Math.}, 35:236--238, 1980.

\bibitem{Stanley96}
R.~P. Stanley.
\newblock {\em Combinatorics and Commutative Algebra}.
\newblock Progress in Mathematics. Birkh{\"a}user, Boston, Inc., Boston, MA,
  1996.
\newblock Second edition.

\bibitem{Swartz05}
E.~Swartz.
\newblock Lower bounds for {$h$}-vectors of {$k$}-{CM}, independence, and
  broken circuit complexes.
\newblock {\em SIAM J. Discrete Math.}, 18(3):647--661, 2004/05.

\bibitem{Tay}
T.-S. Tay.
\newblock Lower-bound theorems for pseudomanifolds.
\newblock {\em Discrete Comput. Geom.}, 13(2):203--216, 1995.

\bibitem{Venturello}
L.~Venturello.
\newblock On flag spheres with few equators.
\newblock arXiv:2203.10003, 2022.

\bibitem{Z-rigidity}
H.~Zheng.
\newblock The rigidity of the graphs of homology spheres minus one edge.
\newblock {\em Discrete Math.}, 343:112135, 6, 2020.

\end{thebibliography}
		\bibliographystyle{plain}
	}
\end{document}